\documentclass[12pt]{amsart}
\usepackage{amssymb}
\usepackage{amsbsy}
\usepackage{amscd}
\usepackage{enumerate}
\usepackage{mathrsfs}
\usepackage{eucal}
\usepackage{mathtools}
\usepackage[hidelinks,unicode]{hyperref}

\numberwithin{equation}{section}
\theoremstyle{plain}
\newtheorem{theorem}{Theorem}[section]
\newtheorem{lemma}[theorem]{Lemma}
\newtheorem{proposition}[theorem]{Proposition}

\theoremstyle{definition}
\newtheorem{definition}[theorem]{Definition}
\newtheorem*{definition*}{Definition}
\theoremstyle{remark}

\newtheorem*{remark*}{Remark}

\newcommand{\Y}{Y_{\hbar}}
\newcommand{\g}{\mathfrak g}
\newcommand{\sfr}{\mathfrak s}
\newcommand{\Tplus}{\mathfrak T_+}
\newcommand{\Bbd}{B^{\mathrm{bd}}}

\newcommand{\gU}{\widehat{\mathfrak g}_U}

\newcommand{\Id}{\operatorname{Id}}
\newcommand{\str}{\operatorname{str}}
\newcommand{\ev}{\operatorname{ev}}

\DeclareUnicodeCharacter{2013}{--}
\DeclareUnicodeCharacter{2014}{---}
\DeclareUnicodeCharacter{2019}{'}

\title[PBW theorem for affine super Yangian]{PBW theorem for affine super Yangian}
\author{Vasiliy Volkov}
\address{Center of Fundamental Mathematics, Moscow Institute of Physics and Technology, Dolgoprudny, Russia}
\email{volkov.vd@phystech.edu}
\author{Vladimir Stukopin}
\address{Moscow Institute of Physics and Technology, Dolgoprudny, Russia; South Mathematical Institute, Vladikavkaz, Russia}
\email{stukopin.va@phystech.edu}
\date{}
\hypersetup{pdftitle={PBW theorem for affine super Yangian},pdfauthor={Vasiliy Volkov and Vladimir Stukopin}}

\begin{document}
\maketitle
\tableofcontents

\section*{Introduction}\label{sec:introduction}

Yangians were introduced by Drinfeld as Hopf algebra quantizations of polynomial current Lie bialgebras \cite{Drinfeld85,Drinfeld88}. PBW theorems identify the classical algebras underlying these quantizations and provide a basis for their representation theory. Levendorskii developed a filtered approach to PBW bases \cite{Levendorskii}. The superalgebraic theory was developed, in particular, in the works of Nazarov and Stukopin \cite{Nazarov91,Stukopin94}. Standard accounts of Yangians and quantum groups can be found in the monographs of Molev and Chari--Pressley \cite{Molev,ChariPressley}.

Affine Yangians involve a further loop direction and require attention to central extensions and completions. The early construction of Boyarchenko and Levendorskii \cite{BoyarchenkoLevendorskii} and the completed coproducts of Guay, Nakajima, and Wendlandt \cite{GNW} form part of this development. Ueda constructed the affine super Yangian and its evaluation homomorphism \cite{UedaConstruction,UedaSurjectivity}. For the PBW problem considered here, the relevant classical object is the universal central extension of a two-variable current Lie superalgebra, including its differential-form central directions.

For finite-type super Yangians, Gow proves a PBW theorem for $Y(\mathfrak{gl}_{m|n})$ and relates its RTT and Drinfeld descriptions through Gauss decomposition \cite{Gow}. The proof separates spanning from linear independence and uses iterated coproducts followed by evaluation homomorphisms. In the ordinary affine setting, Guay, Regelskis, and Wendlandt obtain a PBW theorem for simply laced affine Yangians from vertex operator representations \cite{GRW}. These results illustrate complementary methods for controlling the classical limit. In the present affine super setting, the differential-form central directions and the degreewise completions are retained throughout the coefficientwise separation argument.

This paper proves a PBW theorem for Ueda's distinguished affine super Yangian with $m,n\ge2$, $m\ne n$, and $\varepsilon_1=\varepsilon_2=\frac12$, so that $\hbar=1$. The main result is the graded isomorphism
\[
\operatorname{gr}_F Y_\hbar\cong
U\bigl(\operatorname{uce}(\mathfrak{sl}(m|n)\otimes\mathbb C[z^{\pm1},u])\bigr).
\]
Consequently, any total order on a homogeneous classical basis, homogeneous also in current degree, yields a Yangian basis of ordered monomials in filtered lifts. Even basis elements admit arbitrary nonnegative powers, while odd basis elements occur with exponent zero or one.

The proof first identifies the classical presentation using universal central extensions and Steinberg Lie superalgebras \cite{IoharaKoga,ChenSun}. It then studies the leading symbols of Ueda's evaluation formulas in a completed enveloping algebra. A normal PBW order in the evaluation target gives locally finite coefficients and the continuity needed for the completed calculations. The resulting injective Lie map, together with iterated coproducts and joint separation by tensor powers, proves injectivity of the canonical graded homomorphism. Section~\ref{sec:classical-presentation} establishes the classical identification; Sections~\ref{sec:evaluation-map} and~\ref{sec:coproduct} carry out the evaluation and separation arguments. Appendix~\ref{app:order} chooses an explicit order for $\widehat{\mathfrak{sl}}(3|2)$ and illustrates reordering, central corrections, filtered lifts, and tensor separation.

The logical steps of the isomorphism proof can be located as follows. Surjectivity of the canonical map $\Phi$ is proved in Proposition~\ref{prop:graded-surjection}; after the classical identification of Proposition~\ref{prop:classical-presentation-isomorphism}, this is the map \eqref{eq:canonical-pbw-surjection}. Its injectivity is proved in the concluding proof of Theorem~\ref{thm:pbw} in Section~\ref{sec:coproduct}, using Proposition~\ref{prop:evaluation-leading-symbols} and Lemmas~\ref{lem:iterated-evaluation-symbols} and~\ref{lem:tensor-power-separation}. The final two paragraphs of that proof establish, respectively, completeness and linear independence of the ordered filtered lifts. Here completeness means algebraic spanning by finite linear combinations. The auxiliary classical map $\rho$ is proved surjective in Lemma~\ref{lem:classical-map-extension} and injective in Proposition~\ref{prop:classical-presentation-isomorphism}.

\section{Notation and statement}\label{sec:notation-and-statement}
Let $m,n\ge2$, $m\ne n$, $N=m+n$, and $I=\mathbb Z/N\mathbb Z$, represented by $0,1,\ldots,N-1$. Put
\begin{equation}\label{eq:distinguished-parity-sequence}
p(a)=\begin{cases}
0,&1\le a\le m,\\
1,&m<a\le N,
\end{cases}
\qquad p(0)=p(N)=1,
\end{equation}
and use the cyclic Cartan matrix
\begin{equation}\label{eq:cyclic-cartan-matrix}
a_{ij}=\begin{cases}
(-1)^{p(i)}+(-1)^{p(i+1)},&i=j,\\
-(-1)^{p(i+1)},&j=i+1,\\
-(-1)^{p(i)},&j=i-1,\\
1,&(i,j)=(0,N-1)\text{ or }(N-1,0),\\
0,&\text{otherwise}.
\end{cases}
\end{equation}
The second and third cases in \eqref{eq:cyclic-cartan-matrix} exclude the wrap-around pairs. We use Ueda's convention
\[
\hbar=\varepsilon_1+\varepsilon_2
\]
and, on the equal-parameter line, write
\begin{equation}\label{eq:equal-parameter-yangian}
Y_\hbar:=Y_{\hbar/2,\hbar/2}\bigl(\widehat{\mathfrak{sl}}(m|n)\bigr).
\end{equation}
Throughout the proof $\hbar=1$. Thus \eqref{eq:equal-parameter-yangian} is precisely Ueda's algebra with $\varepsilon_1=\varepsilon_2=\frac12$. The relation with the polynomial parameter family and its other specializations is explained after Theorem~\ref{thm:pbw}.

The bracket is the supercommutator and $\{x,y\}=xy+yx$ is the ordinary anticommutator. The generators $x^\pm_{0,r}$ and $x^\pm_{m,r}$ are odd, and all other standard generators are even. The defining difference relations on the equal-parameter line are
\begin{equation}\label{eq:cartan-current-difference-relation}
[h_{i,r+1},x^\pm_{j,s}]-[h_{i,r},x^\pm_{j,s+1}]
=\pm\frac{\hbar a_{ij}}2\{h_{i,r},x^\pm_{j,s}\},
\end{equation}
\begin{equation}\label{eq:root-current-difference-relation}
[x^\pm_{i,r+1},x^\pm_{j,s}]-[x^\pm_{i,r},x^\pm_{j,s+1}]
=\pm\frac{\hbar a_{ij}}2\{x^\pm_{i,r},x^\pm_{j,s}\}.
\end{equation}
In the present specialization the coefficient on the right is $a_{ij}/2$. The remaining relations are
\begin{equation}\label{eq:yangian-basic-relations}
\begin{aligned}
[h_{i,r},h_{j,s}]&=0,\\
[x^+_{i,r},x^-_{j,s}]&=\delta_{ij}h_{i,r+s},\\
[h_{i,0},x^\pm_{j,s}]&=\pm a_{ij}x^\pm_{j,s}.
\end{aligned}
\end{equation}
together with the symmetrized Serre relations
\begin{equation}\label{eq:yangian-serre-relations}
\sum_{\pi\in\mathfrak S_q}
[x^\pm_{i,r_{\pi(1)}},[x^\pm_{i,r_{\pi(2)}},\ldots,[x^\pm_{i,r_{\pi(q)}},x^\pm_{j,s}]\ldots]]=0,
\quad q=1+|a_{ij}|,\quad i\ne j,
\end{equation}
the isotropic relations $[x^\pm_{i,r},x^\pm_{i,s}]=0$ for $i=0,m$, and the two quartic odd relations
\begin{equation}\label{eq:yangian-quartic-odd-relations}
[[x^\pm_{i-1,r},x^\pm_{i,0}],[x^\pm_{i,0},x^\pm_{i+1,s}]]=0,
\qquad i=0,m.
\end{equation}
Equip $Y_\hbar$ with the increasing current filtration $F_\bullet Y_\hbar$ defined by the generator weights
\begin{equation}\label{eq:current-filtration-weights}
\deg_F x^\pm_{i,r}=r,\qquad \deg_F h_{i,r}=r.
\end{equation}
Explicitly, for $p\ge0$, $F_pY_\hbar$ is the $\mathbb C$-linear span of all finite words in the generators $x^\pm_{i,r}$ and $h_{i,r}$ whose second indices have sum at most $p$, including the empty word $1$. Set $F_pY_\hbar=0$ for $p<0$. This is an exhaustive multiplicative filtration: $F_pY_\hbar\,F_qY_\hbar\subseteq F_{p+q}Y_\hbar$, and
\[
\operatorname{gr}_F Y_\hbar=\bigoplus_{p\ge0}F_pY_\hbar/F_{p-1}Y_\hbar.
\]
In accordance with Ueda's presentation, the elements $x^\pm_{i,r}$ and $h_{i,r}$ will be called the Yangian generators; their second index is their assigned current-filtration weight. Later, for an element $X$ of a matrix Lie superalgebra, the notation $X(a)=X\otimes t^a$ will refer to a loop element of loop degree $a$. Thus the generator index $r$ and the loop degree $a$ are distinct notions throughout this proof. Let
\begin{equation}\label{eq:classical-toroidal-algebra}
\mathfrak g=\mathfrak{sl}(m|n),\qquad A_+=\mathbb C[z^{\pm1},u],\qquad
\mathfrak T_+=\operatorname{uce}(\mathfrak g\otimes A_+).
\end{equation}
The variable $u$ records current degree; $z$ is the affine loop variable (the variable denoted by $t$ in Ueda's loop realization). These two variables have different origins. The affine node already belongs to the loop algebra and contributes $z$ or $z^{-1}$, whereas the generator index $r$ contributes $u^r$. Indeed, after passing to the associated graded, the difference relations become
\begin{equation}\label{eq:graded-current-difference-relations}
\begin{aligned}
[H_{i,r+1},E^\pm_{j,s}]&=[H_{i,r},E^\pm_{j,s+1}],\\
[E^\pm_{i,r+1},E^\pm_{j,s}]&=[E^\pm_{i,r},E^\pm_{j,s+1}],
\end{aligned}
\end{equation}
which is precisely the identity obtained from $u^{r+1}u^s=u^ru^{s+1}$. Thus $u$ is not a deformation parameter; the deformation parameter is $\hbar$, while $u$ records the current degree. The current-degree-zero part of $\mathfrak T_+$ will be identified explicitly with Ueda's affine algebra in \eqref{eq:current-degree-zero-component}--\eqref{eq:current-degree-zero-identification}.

\begin{theorem}[PBW theorem]\label{thm:pbw}
Let $m,n\ge2$ with $m\ne n$, and let
\[
Y_\hbar=Y_{1/2,1/2}\bigl(\widehat{\mathfrak{sl}}(m|n)\bigr)
\]
be the affine super Yangian in the distinguished cyclic presentation, so that $\hbar=1$. For $p\ge0$, let $F_pY_\hbar$ be spanned by words in $x^\pm_{i,r},h_{i,r}$ with sum of second indices at most $p$, including the empty word $1$, and set $F_pY_\hbar=0$ for $p<0$. This defines its current filtration. Let
\[
\mathfrak T_+=\operatorname{uce}\bigl(\mathfrak{sl}(m|n)\otimes\mathbb C[z^{\pm1},u]\bigr)
\]
be the universal central extension, graded by current degree $\deg u=1$, $\deg z=0$, with its induced grading on the centre.
\begin{enumerate}[(i)]
\item The associated graded superalgebra is canonically isomorphic to the enveloping superalgebra of this classical Lie superalgebra:
\begin{equation}\label{eq:pbw-graded-isomorphism}
\operatorname{gr}_F Y_\hbar\cong U\bigl(\operatorname{uce}(\mathfrak{sl}(m|n)\otimes\mathbb C[z^{\pm1},u])\bigr).
\end{equation}
\item Let $B$ be any totally ordered $\mathbb C$-basis of $\mathfrak T_+$ homogeneous in parity and current degree. For every $b\in B$ of current degree $r$, choose an arbitrary lift $\widehat b\in F_rY_\hbar$ whose principal symbol corresponds to $b$ under the isomorphism in part~\textup{(i)}, and fix all these choices. Then the distinct ordered supermonomials in these fixed lifts,
\begin{equation}\label{eq:ordered-filtered-lift-monomials}
\widehat b_1^{\,d_1}\cdots\widehat b_t^{\,d_t},\qquad b_1<\cdots<b_t,
\end{equation}
including the empty monomial $1$, form a $\mathbb C$-vector-space basis of $Y_\hbar$. Here $t\ge0$, $d_i\in\mathbb Z_{\ge0}$ when $b_i$ is even, and $d_i\in\{0,1\}$ when $b_i$ is odd.
\end{enumerate}
\end{theorem}

We now describe the parameter family on the equal-parameter line. Let $\mathcal Y$ be the $\mathbb C[\hbar]$-superalgebra with the generators and relations \eqref{eq:cartan-current-difference-relation}--\eqref{eq:yangian-quartic-odd-relations}, including the isotropic relations, and set $Y_c=\mathcal Y/(\hbar-c)\mathcal Y$ for $c\in\mathbb C$. Thus the normalized algebra used in the proof is $Y_1$. As a consequence of Theorem~\ref{thm:pbw} and the classical identification in Proposition~\ref{prop:classical-presentation-isomorphism},
\[
\mathcal Y\cong\operatorname{Rees}_F(Y_1)
=\bigoplus_{p\ge0}\hbar^p F_pY_1\subseteq Y_1[\hbar],
\]
where a generator of current weight $r$ maps to $\hbar^r$ times the corresponding generator of $Y_1$. Indeed, assigning degree one to $\hbar$ makes the defining relations homogeneous. This map to the Rees algebra is surjective and, by the classical identification and the PBW theorem, is an isomorphism modulo $\hbar$. A homogeneous element in its kernel is therefore divisible by $\hbar$; since $\hbar$ is not a zero divisor in the Rees algebra, division reduces the degree and proves that the kernel is zero. The basis in part~\textup{(ii)}, with each monomial multiplied by $\hbar$ to its total current degree, gives a free $\mathbb C[\hbar]$-basis of the Rees algebra. Consequently $\mathcal Y$ is a flat deformation of the enveloping superalgebra $U(\mathfrak T_+)$, and $Y_0\cong U(\mathfrak T_+)$.

For $c,d\in\mathbb C^\times$, the assignments
\[
x^{\pm,(c)}_{i,r}\longmapsto(c/d)^r x^{\pm,(d)}_{i,r},
\qquad h^{(c)}_{i,r}\longmapsto(c/d)^r h^{(d)}_{i,r}
\]
define an isomorphism $Y_c\longrightarrow Y_d$, with inverse obtained by interchanging $c$ and $d$. These maps preserve the loop grading and intertwine the coproducts with values in the loop-degree completions described in Section~\ref{subsec:completed-tensor-product}. Indeed, the degree-zero generators are fixed, while the correction to the coproduct of a degree-one generator is linear in $\hbar$ \cite[Theorem~4.3, formula~(4.4)]{UedaConstruction}; the factor $c/d$ changes that coefficient from $d$ to $c$. The degree-zero generators and the $h_{i,1}$ generate the algebra, so this verifies compatibility throughout. The counit is also preserved. Thus all nonzero fibers are isomorphic as $\mathbb C$-bialgebras with coproducts valued in these completed tensor products. At $c=0$ the coproduct specializes to the ordinary primitive coproduct on $U(\mathfrak T_+)$. These statements concern the polynomial family with $\varepsilon_1=\varepsilon_2=\hbar/2$.

\section{The classical presentation}\label{sec:classical-presentation}
Let $\mathfrak s$ be the graded Lie superalgebra generated by
\[
E^\pm_{i,r},H_{i,r},\qquad i\in I,\quad r\ge0,
\]
with the parity of the corresponding Yangian generators and with relations
\begin{equation}\label{eq:classical-current-relations}
\begin{aligned}
[H_{i,r},H_{j,s}]&=0,\\
[E^+_{i,r},E^-_{j,s}]&=\delta_{ij}H_{i,r+s},\\
[H_{i,0},E^\pm_{j,s}]&=\pm a_{ij}E^\pm_{j,s},\qquad (i,j\in I,\ r,s\in\mathbb Z_{\ge0}).\\
[H_{i,r+1},E^\pm_{j,s}]&=[H_{i,r},E^\pm_{j,s+1}],\\
[E^\pm_{i,r+1},E^\pm_{j,s}]&=[E^\pm_{i,r},E^\pm_{j,s+1}].
\end{aligned}
\end{equation}
together with the following three homogeneous families. For both choices of sign, for $i\ne j$, put $q=1+|a_{ij}|$. Then, for all $r_1,\ldots,r_q,s\in\mathbb Z_{\ge0}$, impose
\begin{equation}\label{eq:classical-serre-relations}
\sum_{\pi\in\mathfrak S_q}
[E^\pm_{i,r_{\pi(1)}},[E^\pm_{i,r_{\pi(2)}},\ldots,[E^\pm_{i,r_{\pi(q)}},E^\pm_{j,s}]\ldots]]=0.
\end{equation}
For both signs and all $r,s\in\mathbb Z_{\ge0}$, impose the isotropic relations
\begin{equation}\label{eq:classical-isotropic-relations}
[E^\pm_{i,r},E^\pm_{i,s}]=0,\qquad i\in\{0,m\},
\end{equation}
and the quartic odd relations
\begin{equation}\label{eq:classical-quartic-odd-relations}
[[E^\pm_{i-1,r},E^\pm_{i,0}],[E^\pm_{i,0},E^\pm_{i+1,s}]]=0,
\qquad i\in\{0,m\}.
\end{equation}
Here the node subscripts are read in $I=\mathbb Z/N\mathbb Z$. The grading is $\deg E^\pm_{i,r}=\deg H_{i,r}=r$.

\begin{proposition}\label{prop:graded-surjection}
There is a canonical surjective graded homomorphism
\begin{equation}\label{eq:canonical-classical-surjection}
\Phi:U(\mathfrak s)\longrightarrow\operatorname{gr}_F Y_\hbar,
\qquad E^\pm_{i,r}\longmapsto\overline{x^\pm_{i,r}},\quad H_{i,r}\longmapsto\overline{h_{i,r}}.
\end{equation}
\end{proposition}
\begin{proof}
Let $T$ be the free associative superalgebra on the standard generators, graded by their current weights, and let $J$ be the ideal of defining relations. Its grading identifies $\operatorname{gr}_F T$ with $T$. For a nonzero element $f=f_0+\cdots+f_d\in J$, where each $f_j$ is homogeneous of current degree $j$ and $f_d\ne0$, call $f_d$ its leading homogeneous component. The initial ideal $\operatorname{in}_F(J)$ is the two-sided ideal generated by these leading components for \emph{all} nonzero elements of $J$. By contrast, let $J_{\mathrm{top}}$ be generated by the leading components of the displayed defining relations alone. Then
\[
T/J_{\mathrm{top}}=U(\mathfrak s).
\]
Relations \eqref{eq:cartan-current-difference-relation}--\eqref{eq:root-current-difference-relation} have degree $r+s+1$ on the left and degree $r+s$ on the right; every other defining relation is homogeneous. Thus
\[
J_{\mathrm{top}}\subseteq\operatorname{in}_F(J),\qquad
\operatorname{gr}_F(T/J)=T/\operatorname{in}_F(J).
\]
Taking the quotient by the larger initial ideal gives \eqref{eq:canonical-classical-surjection}. \emph{This is the surjectivity step for $\Phi$.} Under the classical identification of Proposition~\ref{prop:classical-presentation-isomorphism}, it gives the surjectivity of \eqref{eq:canonical-pbw-surjection}. Injectivity amounts to showing that the defining relations exhaust the initial ideal; it is proved in the concluding proof of Theorem~\ref{thm:pbw} in Section~\ref{sec:coproduct}. This is the filtered starting point of the strategy used by Levendorskii \cite{Levendorskii}: first construct the canonical graded epimorphism and then prove that it is injective.
\end{proof}

\subsection{The universal central extension and Steinberg algebra}\label{subsec:central-extension-and-steinberg-algebra}
We next identify $\mathfrak s$. We first state in full the two external structural results used in this identification. The algebra $A_+=\mathbb C[z^{\pm1},u]$ is regarded as a purely even commutative superalgebra, and the invariant form on $\mathfrak g=\mathfrak{sl}(m|n)$ is $(X,Y)=\operatorname{str}(XY)$, where $\operatorname{str}(X)=\sum_{a=1}^N(-1)^{p(a)}X_{aa}$. For a parity-homogeneous matrix $X$, write $|X|\in\mathbb Z/2\mathbb Z$ for its parity.

We use the differential-form description of central extensions from Iohara--Koga \cite[Theorem~4.7]{IoharaKoga}. Here $\Omega^1_{A_+/\mathbb C}$ denotes the $A_+$-module of K\"ahler differentials, generated by the symbols $df$ for $f\in A_+$ subject to $\mathbb C$-linearity and the Leibniz relation $d(fg)=f\,dg+g\,df$. The map $d:A_+\to\Omega^1_{A_+/\mathbb C}$ is the universal $\mathbb C$-derivation. Its image $dA_+$ is the $\mathbb C$-linear subspace of exact forms, and $\Omega^1_{A_+/\mathbb C}/dA_+$ is a quotient of $\mathbb C$-vector spaces, placed in even parity. A bar denotes the class modulo exact forms: $\overline\omega=\omega+dA_+$. In particular, $\overline{g\,df}$ is the class of the product of $g\in A_+$ with the differential of $f\in A_+$; this is the notation used in \eqref{eq:universal-central-extension-bracket} and \eqref{eq:differential-cocycle-identities}.

\begin{proposition}[Iohara--Koga]\label{prop:differential-form-central-extension}
For $m,n\ge2$ and $m\ne n$, the universal central extension of $\mathfrak g\otimes A_+$ is
\begin{equation}\label{eq:universal-central-extension-model}
\mathfrak T_+=(\mathfrak g\otimes A_+)\oplus\Omega^1_{A_+/\mathbb C}/dA_+.
\end{equation}
The second summand is central, and for homogeneous $X,Y\in\mathfrak g$ and $f,g\in A_+$ the bracket is
\begin{equation}\label{eq:universal-central-extension-bracket}
[X\otimes f,Y\otimes g]=[X,Y]\otimes fg+\operatorname{str}(XY)\overline{g\,df}.
\end{equation}
The projection
\begin{equation}\label{eq:universal-central-extension-projection}
\pi_{\mathrm{uce}}:\mathfrak T_+\longrightarrow\mathfrak g\otimes A_+,
\qquad X\otimes f+\omega\longmapsto X\otimes f,
\end{equation}
is the universal central extension.
\end{proposition}
\begin{proof}
This is the assertion of Iohara--Koga \cite[Theorem 4.7]{IoharaKoga} specialized to the present data. Its hypotheses hold because $A_+$ is a unital commutative $\mathbb C$-algebra and the ground field has characteristic zero, and $\mathfrak{sl}(m|n)$ with $m\ne n$ is a simple basic classical Lie superalgebra whose supertrace form is nondegenerate. The exceptional type $A(1,1)$ is excluded by $m\ne n$.

For all $f,g,h\in A_+$, the two identities in the quotient $\Omega^1_{A_+/\mathbb C}/dA_+$ which make the displayed cocycle work are
\begin{equation}\label{eq:differential-cocycle-identities}
\overline{g\,df}=-\overline{f\,dg},\qquad
\overline{gh\,df+hf\,dg+fg\,dh}=0.
\end{equation}
The first follows from $d(fg)=f\,dg+g\,df$ and the second from $d(fgh)=gh\,df+hf\,dg+fg\,dh$. Together with
\[
\operatorname{str}(YX)=(-1)^{|X||Y|}\operatorname{str}(XY)
\]
and invariance of the supertrace, these identities verify directly the super-skew-symmetry and the cocycle identity in \eqref{eq:universal-central-extension-bracket}. Universality follows from the cited theorem.
\end{proof}

We make its centre explicit. Since
\begin{equation}\label{eq:kahler-differentials-of-current-ring}
\Omega^1_{A_+/\mathbb C}=A_+\,dz\oplus A_+\,du,\qquad
d(z^ku^r)=kz^{k-1}u^r\,dz+rz^ku^{r-1}\,du,
\end{equation}
both terms on the right have bidegree $(k,r)$. When $r\ge1$, the space of forms of bidegree $(k,r)$ is two-dimensional and its exact subspace is generated by the second expression in \eqref{eq:kahler-differentials-of-current-ring}; the quotient is therefore one-dimensional. When $r=0$, $z^{k-1}\,dz$ is exact for $k\ne0$, whereas $z^{-1}\,dz$ is not exact. Thus
\begin{equation}\label{eq:central-differential-basis}
\frac{\Omega^1_{A_+/\mathbb C}}{dA_+}
=\mathbb C\kappa_{0,0}\oplus\bigoplus_{\substack{k\in\mathbb Z\\r\ge1}}\mathbb C\kappa_{k,r},
\qquad\kappa_{0,0}=\overline{z^{-1}\,dz},\quad
\kappa_{k,r}=\overline{z^{k-1}u^r\,dz}.
\end{equation}
With
\[
\deg u=\deg du=1,\qquad\deg z=\deg dz=0,
\]
the degree-zero component is
\begin{equation}\label{eq:current-degree-zero-component}
(\mathfrak T_+)_0=\mathfrak g\otimes\mathbb C[z^{\pm1}]\oplus\mathbb C\kappa_{0,0}.
\end{equation}
Indeed, \eqref{eq:universal-central-extension-bracket} gives the usual affine bracket
\begin{equation}\label{eq:affine-current-degree-zero-bracket}
[X\otimes z^k,Y\otimes z^\ell]=[X,Y]\otimes z^{k+\ell}
+k\delta_{k+\ell,0}\operatorname{str}(XY)\kappa_{0,0}.
\end{equation}
Consequently, \eqref{eq:current-degree-zero-component} is the derived affine Lie superalgebra, without the degree derivation.

We now spell out the Steinberg result, using the parity convention \eqref{eq:distinguished-parity-sequence}.

\begin{definition}\label{def:steinberg-lie-superalgebra}
The Steinberg Lie superalgebra $\mathfrak{st}_{m|n}(A_+)$ is generated by symbols $F_{ab}(f)$, linear in $f\in A_+$, for $a\ne b$, with parity $p(a)+p(b)$, subject to
\begin{equation}\label{eq:steinberg-chain-relation}
[F_{ab}(f),F_{bc}(g)]=F_{ac}(fg),\qquad a,b,c\text{ pairwise distinct},
\end{equation}
\begin{equation}\label{eq:steinberg-disjoint-relation}
[F_{ab}(f),F_{cd}(g)]=0,\qquad b\ne c,\quad a\ne d.
\end{equation}
\end{definition}

\begin{theorem}\label{thm:steinberg-universal-central-extension}
Let $m,n\ge2$ with $m\ne n$, put $N=m+n$, and let $A_+=\mathbb C[z^{\pm1},u]$ be purely even. Let $\mathfrak{sl}_{m|n}(A_+)$ be the Lie superalgebra of supertrace-zero $N\times N$ matrices over $A_+$, with matrix-unit parity $p(a)+p(b)$, where $p(a)=0$ for $a\le m$ and $p(a)=1$ for $a>m$. Let $\mathfrak{st}_{m|n}(A_+)$ be generated by the symbols $F_{ab}(f)$ for $a\ne b$ and $f\in A_+$, linear in $f$ and of parity $p(a)+p(b)$, subject to
\[
\begin{aligned}
[F_{ab}(f),F_{bc}(g)]&=F_{ac}(fg)&&\text{for distinct }a,b,c,\\
[F_{ab}(f),F_{cd}(g)]&=0&&\text{if }b\ne c\text{ and }a\ne d.
\end{aligned}
\]
Then the homomorphism
\begin{equation}\label{eq:steinberg-central-extension}
\varphi:\mathfrak{st}_{m|n}(A_+)\longrightarrow\mathfrak{sl}_{m|n}(A_+),
\qquad F_{ab}(f)\longmapsto E_{ab}\otimes f,
\end{equation}
is surjective with central kernel canonically isomorphic to the first cyclic homology $HC_1(A_+)$. It is the universal central extension: for every central extension $q:\mathfrak e\to\mathfrak{sl}_{m|n}(A_+)$ there is a unique Lie-superalgebra homomorphism $\theta:\mathfrak{st}_{m|n}(A_+)\to\mathfrak e$ such that $q\circ\theta=\varphi$.
\end{theorem}

\begin{remark*}
The central-extension assertion is Chen--Sun's Proposition~4.1, the kernel identification is their Theorem~5.1, and universality is their Theorem~6.1 \cite[Definition~4.1, Proposition~4.1 and Theorems~5.1, 6.1]{ChenSun}. Their hypotheses hold here: $A_+$ is purely even, associative, unital and has the $\mathbb C$-basis $\{z^ku^r\mid k\in\mathbb Z,r\ge0\}$ containing $1$, while $m,n\ge2$ and $m\ne n$ imply $N=m+n\ge5$.
\end{remark*}

Since $A_+$ is commutative,
\begin{equation}\label{eq:special-linear-current-algebra}
\mathfrak{sl}_{m|n}(A_+)=\mathfrak{sl}(m|n)\otimes A_+.
\end{equation}
For this purely even commutative algebra, the first cyclic homology $HC_1(A_+)$ is the quotient of $A_+\otimes_{\mathbb C}A_+$ by the $\mathbb C$-linear span of
\[
f\otimes g+g\otimes f,\qquad
f\otimes gh+g\otimes hf+h\otimes fg,\qquad f,g,h\in A_+,
\]
as in the tensor description of Chen--Sun \cite[Definition~4.1 and Theorem~5.1]{ChenSun}. Denote the class of $f\otimes g$ in this quotient by $\langle\!\langle f,g\rangle\!\rangle$. Its defining relations, for all $f,g,h\in A_+$, are
\begin{equation}\label{eq:cyclic-homology-skew-symmetry}
\langle\!\langle f,g\rangle\!\rangle+\langle\!\langle g,f\rangle\!\rangle=0,
\end{equation}
\begin{equation}\label{eq:cyclic-homology-cycle-relation}
\langle\!\langle f,gh\rangle\!\rangle+
\langle\!\langle g,hf\rangle\!\rangle+
\langle\!\langle h,fg\rangle\!\rangle=0.
\end{equation}
The differential description used above identifies this space with the central kernel in the Iohara--Koga model \cite[Theorem~4.7]{IoharaKoga}:
\begin{equation}\label{eq:cyclic-homology-differential-isomorphism}
HC_1(A_+)\xrightarrow{\sim}\Omega^1_{A_+/\mathbb C}/dA_+,
\qquad\langle\!\langle f,g\rangle\!\rangle\longmapsto\overline{g\,df}.
\end{equation}
Relations \eqref{eq:cyclic-homology-skew-symmetry}--\eqref{eq:cyclic-homology-cycle-relation} map respectively to $\overline{d(fg)}=0$ and $\overline{d(fgh)}=0$; conversely they are exactly the antisymmetry and Leibniz relations needed to define the inverse $\overline{g\,df}\mapsto\langle\!\langle f,g\rangle\!\rangle$.

The two universal central extensions are identified by
\begin{equation}\label{eq:steinberg-differential-isomorphism}
\psi:\mathfrak{st}_{m|n}(A_+)\xrightarrow{\sim}\mathfrak T_+,
\qquad F_{ab}(f)\longmapsto E_{ab}\otimes f.
\end{equation}
Indeed, if $a,b,c$ are distinct, then $E_{ab}E_{bc}=E_{ac}$, $E_{bc}E_{ab}=0$, and $\operatorname{str}(E_{ab}E_{bc})=0$, so the first Steinberg relation is preserved. If $b\ne c$ and $a\ne d$, both products of $E_{ab}$ and $E_{cd}$ and their supertrace pairing vanish, so the second relation is preserved. On the central kernel, fix $j\ne1$ and set
\[
H_{1j}(f,g)=[F_{1j}(f),F_{j1}(g)],\qquad h_j(f,g)=H_{1j}(f,g)-H_{1j}(1,gf).
\]
Then a direct use of \eqref{eq:universal-central-extension-bracket} gives
\begin{equation}\label{eq:steinberg-diagonal-images}
\begin{aligned}
\psi(H_{1j}(f,g))&=(E_{11}-(-1)^{p(j)}E_{jj})\otimes fg+\overline{g\,df},\\
\psi(H_{1j}(1,gf))&=(E_{11}-(-1)^{p(j)}E_{jj})\otimes gf+\overline{gf\,d1}\\
&=(E_{11}-(-1)^{p(j)}E_{jj})\otimes fg.
\end{aligned}
\end{equation}
Subtracting the two lines yields
\begin{equation}\label{eq:steinberg-central-images}
\psi(h_j(f,g))=\overline{g\,df},
\end{equation}
which agrees with \eqref{eq:cyclic-homology-differential-isomorphism}. Finally, both sides of \eqref{eq:steinberg-differential-isomorphism} are universal central extensions of the same current Lie superalgebra, so uniqueness of a universal central extension makes $\psi$ an isomorphism.

Consider the following assignment from the standard generators of $\mathfrak s$ to $\mathfrak T_+$. For $1\le i<N$, prescribe
\begin{equation}\label{eq:finite-generator-assignment}
\begin{aligned}
E^+_{i,r}&\longmapsto E_{i,i+1}\otimes u^r,\\
E^-_{i,r}&\longmapsto(-1)^{p(i)}E_{i+1,i}\otimes u^r,\\
H_{i,r}&\longmapsto\bigl((-1)^{p(i)}E_{ii}-(-1)^{p(i+1)}E_{i+1,i+1}\bigr)\otimes u^r.
\end{aligned}
\end{equation}
and at the affine vertex prescribe
\begin{equation}\label{eq:affine-generator-assignment}
\begin{aligned}
E^+_{0,r}&\longmapsto E_{N1}\otimes zu^r,\\
E^-_{0,r}&\longmapsto-E_{1N}\otimes z^{-1}u^r,\\
H_{0,r}&\longmapsto(-E_{11}-E_{NN})\otimes u^r+\overline{z^{-1}u^r\,dz}.
\end{aligned}
\end{equation}

\subsection{The affine algebra in current degree zero}\label{subsec:current-degree-zero-affine-algebra}
Here ``current degree zero'' means degree zero in the variable $u$. It places no restriction on the exponent of the affine loop variable $z$: the affine generators already contain $z$ and $z^{-1}$, and their brackets produce loop elements with arbitrary integral powers of $z$.

We record explicitly the affine presentation used below.

Let $\mathfrak a_{\mathrm{CS}}$ denote the Chevalley--Serre presentation used in
\cite[Definition~2.3 and Proposition~2.4]{UedaConstruction}, with the loop variable renamed $z$.
Its generators are $e_i^\pm,k_i$ for $i\in I$, with $e_0^\pm,e_m^\pm$ odd,
and its relations are
\begin{equation}
\begin{aligned}
 &[k_i,k_j]=0,\qquad [k_i,e_j^\pm]=\pm a_{ij}e_j^\pm,
       \qquad [e_i^+,e_j^-]=\delta_{ij}k_i,\\
 &(\operatorname{ad}e_i^\pm)^{1+|a_{ij}|}(e_j^\pm)=0
       &&(i\ne j),\\
 &[e_i^\pm,e_i^\pm]=0 &&(i=0,m),\\
 &[[e_{i-1}^\pm,e_i^\pm],[e_{i+1}^\pm,e_i^\pm]]=0
       &&(i=0,m).
\end{aligned}
\label{eq:affine-chevalley-relations}
\end{equation}
We use the following concrete affine loop notation:
\begin{equation}
\widehat{\g}_{\mathrm U}=\g\otimes\mathbb C[z^{\pm1}]\oplus\mathbb Cc,
\qquad
[Xz^r,Yz^s]=[X,Y]z^{r+s}
   +r\delta_{r+s,0}\str(XY)c.
\label{eq:affine-loop-model}
\end{equation}
Write $\Xi$ for the Chevalley generator assignment with normalization
\begin{equation}
\begin{aligned}
 e_i^+&\longmapsto E_{i,i+1},
 &e_i^-&\longmapsto(-1)^{p(i)}E_{i+1,i}
       &&(1\le i<N),\\
 k_i&\longmapsto(-1)^{p(i)}E_{ii}
                    -(-1)^{p(i+1)}E_{i+1,i+1}
       &&&&(1\le i<N),\\
 e_0^+&\longmapsto E_{N1}z,
 &e_0^-&\longmapsto-E_{1N}z^{-1},\\
 k_0&\longmapsto-E_{11}-E_{NN}+c.
\end{aligned}
\label{eq:affine-chevalley-normalization}
\end{equation}
In Ueda's notation, \eqref{eq:affine-loop-model} is $\widehat{\g}_{\mathrm U}$; the full affine algebra
also contains the degree derivation:
\begin{equation}
\widetilde{\g}_{\mathrm U}
   =\widehat{\g}_{\mathrm U}\oplus\mathbb Cd,
\qquad [d,Xz^r]=rXz^r.
\label{eq:affine-degree-derivation}
\end{equation}
The quartic relation in \eqref{eq:affine-chevalley-relations} uses the order $[e_{i+1}^\pm,e_i^\pm]$
in its second inner bracket. Reversing this order gives the convention in
\eqref{eq:classical-quartic-odd-relations}, up to a nonzero super-skew-symmetry sign. The central term at the
affine vertex can be checked directly:
\begin{equation}
\begin{aligned}
 [E_{N1}z,-E_{1N}z^{-1}]
   &=(-E_{11}-E_{NN})+\str(E_{N1}(-E_{1N}))c\\
   &=-E_{11}-E_{NN}+c,
\end{aligned}
\label{eq:affine-node-central-bracket}
\end{equation}
because both matrix units are odd and $\str(-E_{NN})=1$. At a finite vertex,
\begin{equation}
 [E_{i,i+1},(-1)^{p(i)}E_{i+1,i}]
   =(-1)^{p(i)}E_{ii}-(-1)^{p(i+1)}E_{i+1,i+1}.
\label{eq:finite-node-chevalley-bracket}
\end{equation}
The remaining matrix checks use
\begin{equation}
 [E_{ab},E_{cd}]
  =\delta_{bc}E_{ad}
    -(-1)^{(p(a)+p(b))(p(c)+p(d))}\delta_{ad}E_{cb}.
\label{eq:matrix-unit-superbracket}
\end{equation}
Since $\g$ is perfect, its loop algebra lies in the derived algebra of
$\widetilde{\g}_{\mathrm U}$. Both the left-hand side of \eqref{eq:affine-node-central-bracket} and its
diagonal term are derived, so $c$ is derived as well. No commutator has a
$d$-component. Hence
\begin{equation}
 [\widetilde{\g}_{\mathrm U},\widetilde{\g}_{\mathrm U}]
   =\widehat{\g}_{\mathrm U}.
\label{eq:affine-derived-subalgebra}
\end{equation}

\begin{lemma}\label{lem:current-degree-zero-identification}
There are homomorphisms
\begin{equation}
 \mathfrak a_{\mathrm{CS}}
   \xrightarrow{\ \iota_0\ }\mathfrak s
   \xrightarrow{\ \pi_0\ }\mathfrak a_{\mathrm{CS}}
\label{eq:current-degree-zero-retraction}
\end{equation}
given on generators by
\begin{equation}
\begin{aligned}
 \iota_0(e_i^\pm)&=E_{i,0}^\pm,
 &\iota_0(k_i)&=H_{i,0},\\
 \pi_0(E_{i,0}^\pm)&=e_i^\pm,
 &\pi_0(H_{i,0})&=k_i,\\
 \pi_0(E_{i,r}^\pm)&=0,
 &\pi_0(H_{i,r})&=0\qquad(r>0).
\end{aligned}
\label{eq:current-degree-zero-retraction-generators}
\end{equation}
Moreover, $\pi_0\iota_0=\Id$ and
\begin{equation}
 \mathfrak s_0\cong\mathfrak a_{\mathrm{CS}}
       \cong\widehat{\g}_{\mathrm U}.
\label{eq:current-degree-zero-identification}
\end{equation}
\end{lemma}

\begin{proof}
Putting every current index equal to zero in \eqref{eq:classical-current-relations} gives the first three
relations of \eqref{eq:affine-chevalley-relations}. A symmetrized Serre relation then equals
\begin{equation}
 (1+|a_{ij}|)!
 (\operatorname{ad}E_{i,0}^\pm)^{1+|a_{ij}|}(E_{j,0}^\pm)=0.
\label{eq:zero-mode-serre-specialization}
\end{equation}
The isotropic relations agree, and the quartic relations agree after the
super-skew-symmetry adjustment described above. Thus $\iota_0$ is well defined.

For $\pi_0$, the root--coroot relation maps to
\begin{equation}
 \pi_0([E_{i,r}^+,E_{j,s}^-]-\delta_{ij}H_{i,r+s})
 =\begin{cases}
 [e_i^+,e_j^-]-\delta_{ij}k_i,&r=s=0,\\
 0,&r+s>0,
 \end{cases}
\label{eq:retraction-root-bracket-check}
\end{equation}
and both expressions vanish. Similarly,
\begin{equation}
 \pi_0([H_{i,0},E_{j,s}^\pm]\mp a_{ij}E_{j,s}^\pm)
 =\begin{cases}
 [k_i,e_j^\pm]\mp a_{ij}e_j^\pm,&s=0,\\
 0,&s>0.
 \end{cases}
\label{eq:retraction-cartan-action-check}
\end{equation}
Every term of a difference relation contains a generator of positive
current degree, so its image is zero. A Serre, isotropic, or quartic relation
maps either to its current-degree-zero counterpart or to zero. Hence
$\pi_0$ is well defined and $\pi_0\iota_0=\Id$.

Finally, $\mathfrak s$ is a quotient of a nonnegatively graded free Lie
superalgebra by homogeneous relations. A degree-zero Lie word uses only
degree-zero generators, so $\mathfrak s_0=\operatorname{im}\iota_0$ and
$\mathfrak s_0\cong\mathfrak a_{\mathrm{CS}}$. We now use
\cite[Definition~2.3 and Proposition~2.4]{UedaConstruction} to identify this Chevalley--Serre
algebra with the concrete loop algebra \eqref{eq:affine-loop-model} through the normalized
assignment \eqref{eq:affine-chevalley-normalization}. This gives the remaining isomorphism in \eqref{eq:current-degree-zero-identification}.
\end{proof}

Applying the enveloping functor gives
\begin{equation}
 U(\mathfrak a_{\mathrm{CS}})
  \xrightarrow{\ U(\iota_0)\ }U(\mathfrak s)
  \xrightarrow{\ U(\pi_0)\ }U(\mathfrak a_{\mathrm{CS}}),
 \qquad U(\pi_0)U(\iota_0)=\Id.
\label{eq:enveloping-retraction}
\end{equation}
Equation \eqref{eq:enveloping-retraction} makes $U(\iota_0)$ a split injection. Thus the enveloping
algebra of the current-degree-zero subalgebra embeds into $U(\mathfrak s)$. By Lemma~\ref{lem:current-degree-zero-identification}, $\iota_0$ is an isomorphism from $\mathfrak a_{\mathrm{CS}}$ onto $\mathfrak s_0$. Identify $\widehat{\g}_{\mathrm U}$ with $(\mathfrak T_+)_0$ by $Xz^k\mapsto X\otimes z^k$ and $c\mapsto\kappa_{0,0}=\overline{z^{-1}dz}$. Under this identification, the composition
\begin{equation}
 \Xi\circ\iota_0^{-1}:\mathfrak s_0\xrightarrow{\sim}(\mathfrak T_+)_0
 \hookrightarrow\mathfrak T_+
\label{eq:rho-current-degree-zero-restriction}
\end{equation}
agrees on the current-degree-zero generators with the assignment \eqref{eq:finite-generator-assignment}--\eqref{eq:affine-generator-assignment}.

\begin{lemma}\label{lem:classical-map-extension}
The assignments \eqref{eq:finite-generator-assignment}--\eqref{eq:affine-generator-assignment} extend uniquely to a surjective graded Lie superalgebra homomorphism
$\rho:\mathfrak s\twoheadrightarrow\Tplus$. Its restriction to $\mathfrak s_0$ is the map \eqref{eq:rho-current-degree-zero-restriction}.
\end{lemma}

\begin{proof}
Introduce the matrix representatives of the simple Chevalley generators:
\begin{equation}
\begin{aligned}
 D_i&=(-1)^{p(i)}E_{ii}-(-1)^{p(i+1)}E_{i+1,i+1},
 &X_i^+&=E_{i,i+1},
 &X_i^-&=(-1)^{p(i)}E_{i+1,i}\quad(1\le i<N),\\
 D_0&=-E_{11}-E_{NN},
 &X_0^+&=E_{N1},
 &X_0^-&=-E_{1N},\\
 \ell_i^\pm&=0\quad(1\le i<N),
 &\ell_0^+&=1,
 &\ell_0^-&=-1.
\end{aligned}
\label{eq:rho-matrix-and-loop-data}
\end{equation}
The assignments \eqref{eq:finite-generator-assignment}--\eqref{eq:affine-generator-assignment} are therefore
\begin{equation}
 \rho(E_{i,r}^\pm)=X_i^\pm z^{\ell_i^\pm}u^r,
 \qquad
 \rho(H_{i,r})=D_i u^r+\delta_{i0}z^{-1}u^r dz.
\label{eq:rho-generator-formulas}
\end{equation}
We repeatedly use the elementary identities
\begin{equation}
 E_{ab}E_{cd}=\delta_{bc}E_{ad},
 \qquad
 \str(E_{ab}E_{cd})=\delta_{bc}\delta_{ad}(-1)^{p(a)},
\label{eq:matrix-unit-product-supertrace}
\end{equation}
and the superbracket formula \eqref{eq:matrix-unit-superbracket}.

\noindent For the Cartan--Cartan relations, the $D_i$ are even and diagonal, so $[D_i,D_j]=0$. The only possible nonzero term is the cocycle
\begin{equation}
\begin{aligned}
[D_i u^r,D_j u^s]
 &=\str(D_iD_j)\,\overline{u^s\,d(u^r)}\\
 &=r\,\str(D_iD_j)\,\overline{u^{r+s-1}\,du}=0.
\end{aligned}
\label{eq:rho-cartan-commutativity}
\end{equation}
When $r+s>0$, the last class is $\frac{r}{r+s}\overline{d(u^{r+s})}=0$; when $r+s=0$, one has $r=s=0$. The differential part of $\rho(H_{0,r})$ is central.

For $1\le i<N$, direct multiplication gives
\begin{equation}
\begin{aligned}
[X_i^+,X_i^-]&=(-1)^{p(i)}E_{ii}-(-1)^{p(i+1)}E_{i+1,i+1}=D_i,\\
\str(X_i^+X_i^-)&=(-1)^{2p(i)}=1.
\end{aligned}
\label{eq:finite-simple-matrix-pairing}
\end{equation}
Consequently,
\begin{equation}
\begin{aligned}
[\rho(E_{i,r}^+),\rho(E_{i,s}^-)]
 &=D_i u^{r+s}+r\,\overline{u^{r+s-1}\,du}\\
 &=D_i u^{r+s}=\rho(H_{i,r+s}),
\end{aligned}
\label{eq:rho-finite-opposite-root-bracket}
\end{equation}
where exactness was just checked. At the affine vertex, for $f=zu^r$ and $g=z^{-1}u^s$,
\begin{equation}
\begin{aligned}
\overline{g\,df}
 &=\overline{z^{-1}u^{r+s}\,dz}+r\,\overline{u^{r+s-1}\,du}\\
 &=\overline{z^{-1}u^{r+s}\,dz}.
\end{aligned}
\label{eq:affine-node-differential-cocycle}
\end{equation}
Since $X_0^{\pm}$ are odd,
\begin{equation}
\begin{aligned}
[X_0^+,X_0^-]&=[E_{N1},-E_{1N}]=-E_{NN}-E_{11}=D_0,\\
\str(X_0^+X_0^-)&=\str(-E_{NN})=1.
\end{aligned}
\label{eq:affine-simple-matrix-pairing}
\end{equation}
Thus
\begin{equation}
[\rho(E_{0,r}^+),\rho(E_{0,s}^-)]
 =D_0u^{r+s}+\overline{z^{-1}u^{r+s}\,dz}=\rho(H_{0,r+s}).
\label{eq:rho-affine-opposite-root-bracket}
\end{equation}
If $i\ne j$, both products $X_i^+X_j^-$ and $X_j^-X_i^+$ vanish, as does their supertrace pairing. Hence
\begin{equation}
[\rho(E_{i,r}^+),\rho(E_{j,s}^-)]=0\qquad(i\ne j).
\label{eq:rho-distinct-opposite-root-bracket}
\end{equation}
We compute the Cartan action explicitly. Let $d_i(a)$ be the coefficient of $E_{aa}$ in $D_i$. Then
\begin{equation}
[D_i,E_{ab}]=(d_i(a)-d_i(b))E_{ab}.
\label{eq:diagonal-matrix-action}
\end{equation}
The definition \eqref{eq:cyclic-cartan-matrix} gives
\begin{equation}
\begin{aligned}
d_i(j)-d_i(j+1)&=a_{ij}\qquad(1\le j<N),\\
d_i(N)-d_i(1)&=a_{i0}.
\end{aligned}
\label{eq:cartan-matrix-diagonal-differences}
\end{equation}
For the wrap-around cases this reads explicitly
\[
\begin{aligned}
d_0(1)-d_0(2)&=-1=a_{01}, &d_0(N-1)-d_0(N)&=1=a_{0,N-1},\\
d_1(N)-d_1(1)&=-1=a_{10}, &d_{N-1}(N)-d_{N-1}(1)&=1=a_{N-1,0}.
\end{aligned}
\]
A diagonal matrix is supertrace-orthogonal to every off-diagonal matrix, so no cocycle occurs. Therefore
\begin{equation}
[\rho(H_{i,0}),\rho(E_{j,s}^{\pm})]=\pm a_{ij}\rho(E_{j,s}^{\pm}).
\label{eq:rho-zero-mode-cartan-action}
\end{equation}
The same calculation proves both current-index shift relations. The central part of $\rho(H_{0,r})$ commutes with everything, and the relevant supertrace pairings vanish, so
\begin{equation}
\begin{aligned}
[\rho(H_{i,r+1}),\rho(E_{j,s}^{\pm})]
 &=\pm a_{ij}X_j^{\pm}z^{\ell_j^{\pm}}u^{r+s+1}\\
 &=[\rho(H_{i,r}),\rho(E_{j,s+1}^{\pm})].
\end{aligned}
\label{eq:rho-cartan-current-difference}
\end{equation}
Also $\str(X_i^{\pm}X_j^{\pm})=0$ for all simple vertices. Hence
\begin{equation}
\begin{aligned}
[\rho(E_{i,r+1}^{\pm}),\rho(E_{j,s}^{\pm})]
 &=[X_i^{\pm},X_j^{\pm}]z^{\ell_i^{\pm}+\ell_j^{\pm}}u^{r+s+1}\\
 &=[\rho(E_{i,r}^{\pm}),\rho(E_{j,s+1}^{\pm})].
\end{aligned}
\label{eq:rho-root-current-difference}
\end{equation}
It remains to check the three higher relation families. If $a_{ij}=0$, then $[X_i^{\pm},X_j^{\pm}]=0$. At adjacent finite vertices,
\begin{equation}
[E_{i,i+1},E_{i+1,i+2}]=E_{i,i+2},\qquad [E_{i,i+1},E_{i,i+2}]=0,
\label{eq:finite-serre-matrix-check}
\end{equation}
while at the two affine edges
\begin{equation}
\begin{aligned}
[E_{N1},E_{12}]&=E_{N2}, &[E_{N1},E_{N2}]&=0,\\
[E_{N-1,N},E_{N1}]&=E_{N-1,1}, &[E_{N-1,N},E_{N-1,1}]&=0.
\end{aligned}
\label{eq:affine-serre-matrix-check}
\end{equation}
The reversed orientations and negative generators give the transposed identities, up to the scalar signs in \eqref{eq:rho-matrix-and-loop-data}; the second bracket remains zero. Every supertrace pairing here is zero, and scalar monomials only multiply the identities. Thus each summand in \eqref{eq:classical-serre-relations} maps to zero.

For the two isotropic vertices the root matrices are odd and satisfy
\begin{equation}
(X_0^{\pm})^2=(X_m^{\pm})^2=0,\qquad
\str((X_0^{\pm})^2)=\str((X_m^{\pm})^2)=0.
\label{eq:isotropic-matrix-squares}
\end{equation}
For scalar monomials $f,g$ this gives
\begin{equation}
[X_i^{\pm}f,X_i^{\pm}g]
 =2(X_i^{\pm})^2fg+\str((X_i^{\pm})^2)\,\overline{g\,df}=0,
\qquad i=0,m.
\label{eq:rho-isotropic-relation}
\end{equation}
At the finite odd vertex, the positive quartic relation becomes
\begin{equation}
\begin{aligned}
[X_{m-1}^+,X_m^+]&=E_{m-1,m+1}, &[X_m^+,X_{m+1}^+]&=E_{m,m+2},\\
[E_{m-1,m+1},E_{m,m+2}]&=0, &\str(E_{m-1,m+1}E_{m,m+2})&=0.
\end{aligned}
\label{eq:finite-positive-quartic-check}
\end{equation}
At the affine odd vertex it becomes
\begin{equation}
\begin{aligned}
[X_{N-1}^+,X_0^+]&=E_{N-1,1}, &[X_0^+,X_1^+]&=E_{N2},\\
[E_{N-1,1},E_{N2}]&=0, &\str(E_{N-1,1}E_{N2})&=0.
\end{aligned}
\label{eq:affine-positive-quartic-check}
\end{equation}
The omitted $z$-factor in each affine inner bracket is the same and does not affect the matrix calculation. For the negative generators,
\begin{equation}
\begin{aligned}
[X_{m-1}^-,X_m^-]&=-E_{m+1,m-1}, &[X_m^-,X_{m+1}^-]&=E_{m+2,m},\\
[X_{N-1}^-,X_0^-]&=-E_{1,N-1}, &[X_0^-,X_1^-]&=E_{2N}.
\end{aligned}
\label{eq:negative-quartic-matrix-brackets}
\end{equation}
In each row the two matrix units have zero product in both orders and zero supertrace pairing. Restoring the powers of $u$ and $z^{\pm1}$ only multiplies these zero identities. Thus the prescribed images satisfy all defining relations \eqref{eq:classical-current-relations}--\eqref{eq:classical-quartic-odd-relations}, so the assignment extends to a Lie superalgebra homomorphism $\rho:\mathfrak s\to\mathfrak T_+$. This extension is unique because the elements $E^\pm_{i,r},H_{i,r}$ generate $\mathfrak s$. Their prescribed images have current degree $r$, so $\rho$ is graded. On $\mathfrak s_0$ it agrees with \eqref{eq:rho-current-degree-zero-restriction}, since both maps have the same values on the current-degree-zero generators.

Let $q:\Tplus\to\g\otimes A_+$ be the central quotient and $S=\operatorname{im}\rho$. By \eqref{eq:rho-current-degree-zero-restriction}, the current-degree-zero generators give the entire affine algebra, so $q(S)$ contains $\g\otimes\mathbb C[z^{\pm1}]$. For fixed $r>0$, the intersection of $q(S)$ with $\g\otimes\mathbb C[z^{\pm1}]u^r$ is stable under the adjoint action of this loop algebra and contains a nonzero simple-root vector. Bracketing with constant matrices generates all of $\g$ because $\g$ is simple, and bracketing with loop matrices supplies every power of $z$. Therefore
\begin{equation}
q(S)\supseteq\g\otimes\mathbb C[z^{\pm1}]u^r\qquad(r\ge0),
\label{eq:rho-image-current-components}
\end{equation}
and hence $q(S)=\g\otimes A_+$ and $\Tplus=S+Z(\Tplus)$.

Perfectness can also be seen explicitly. If $X=[A,B]\in\g$, then
\begin{equation}
[A\otimes1,B\otimes f]=X\otimes f,
\label{eq:perfect-current-algebra-generation}
\end{equation}
so the current summand is derived. Choose commuting even Cartan elements $D,D'$ with $\str(DD')=1$; this is possible by nondegeneracy of the Cartan form. Then
\begin{equation}
[D\otimes f,D'\otimes g]=\overline{g\,df}.
\label{eq:central-differential-generation}
\end{equation}
These classes span $\Omega^1_{A_+/\mathbb C}/dA_+$, so the centre is derived as well. Thus $\Tplus$ is perfect, and
\[
\Tplus=[\Tplus,\Tplus]
 =[S+Z(\Tplus),S+Z(\Tplus)]=[S,S]\subseteq S.
\]
Therefore $S=\Tplus$. This proves the surjectivity of $\rho$; its injectivity will be established by the inverse-map construction in Proposition~\ref{prop:classical-presentation-isomorphism}.
\end{proof}

To prove injectivity of $\rho$, we use the standard description of the real affine roots and the commutator relation established in Lemma~\ref{lem:real-root-cartan-identity}. We first fix the root conventions and prove separatedness of the degreewise completion.

\medskip
\noindent\textbf{Affine-root conventions.} Let
\[
\mathfrak h=\left\{D=\operatorname{diag}(d_1,\ldots,d_N):
 \sum_{a=1}^{N}(-1)^{p(a)}d_a=0\right\}
\]
be the finite Cartan subalgebra. Write $\varepsilon_a(D)=d_a$, let $\delta$ vanish on $\mathfrak h\oplus\mathbb Cc$, and put $\delta(d)=1$. With the loop variable $z$ used in \eqref{eq:affine-loop-model}, the simple affine roots and the real roots are
\begin{equation}
\begin{gathered}
\alpha_i=\varepsilon_i-\varepsilon_{i+1}\quad(1\le i<N),
\qquad \alpha_0=\delta+\varepsilon_N-\varepsilon_1,\\
\Delta^{\mathrm{re}}=\{\varepsilon_a-\varepsilon_b+k\delta:a\ne b,\ k\in\mathbb Z\}.
\end{gathered}
\label{eq:affine-real-roots}
\end{equation}
Indeed, for $D=\operatorname{diag}(d_1,\ldots,d_N)$,
\begin{equation}
[D,E_{ab}z^k]=(d_a-d_b)E_{ab}z^k,\qquad
[d,E_{ab}z^k]=kE_{ab}z^k,
\label{eq:affine-root-weight-actions}
\end{equation}
so $E_{ab}z^k$ spans the root space of $\varepsilon_a-\varepsilon_b+k\delta$. We use the positive system
\begin{equation}
\begin{aligned}
\Delta^{\mathrm{re}}_+
 &=\{\varepsilon_a-\varepsilon_b+k\delta:k>0,\ a\ne b\}
 \sqcup\{\varepsilon_a-\varepsilon_b:a<b\},\\
\Delta_+&=\Delta^{\mathrm{re}}_+\sqcup\{k\delta:k>0\}.
\end{aligned}
\label{eq:positive-affine-roots}
\end{equation}
The symmetric root-lattice form used below is characterized by $(\alpha_i,\alpha_j)=a_{ij}$; explicitly on finite root differences,
\begin{equation}
\begin{gathered}
(\varepsilon_a-\varepsilon_b,\varepsilon_c-\varepsilon_d)
 =(-1)^{p(a)}(\delta_{ac}-\delta_{ad})-(-1)^{p(b)}(\delta_{bc}-\delta_{bd}),\\
(\delta,\delta)=(\delta,\varepsilon_a-\varepsilon_b)=0.
\end{gathered}
\label{eq:affine-root-bilinear-form}
\end{equation}
The invariant affine form is normalized by
\begin{equation}
\begin{gathered}
\kappa(Xz^r,Yz^s)=\delta_{r+s,0}\str(XY),\qquad \kappa(c,d)=1,\\
\kappa(c,c)=\kappa(d,d)=0,
\end{gathered}
\label{eq:affine-invariant-form}
\end{equation}
with $c,d$ orthogonal to the loop part. In particular, $(-1)^{p(a)}E_{ba}z^{-k}$ is dual to $E_{ab}z^k$. Thus all root pairings used in \cite[Proposition 4.26]{UedaConstruction} are fixed by \eqref{eq:affine-real-roots}--\eqref{eq:affine-invariant-form}.

\begin{lemma}\label{lem:degreewise-completion-separatedness}
Let
\begin{equation}
\mathfrak a=\bigoplus_{q\in\mathbb Z}\mathfrak a(q),\qquad
[\mathfrak a(r),\mathfrak a(s)]\subseteq\mathfrak a(r+s),
\label{eq:graded-lie-algebra}
\end{equation}
be a graded Lie superalgebra, and let $A=U(\mathfrak a)$ with its induced grading. For $L\in\mathbb Z_{\ge0}$, set
\begin{equation}
T_LA(q)=\sum_{s>L}A(q-s)A(s),\qquad
\widehat A(q)=\varprojlim_L A(q)/T_LA(q).
\label{eq:right-tail-degreewise-completion}
\end{equation}
Here the sum means the linear span of the indicated products, and the inverse limit is taken as $L$ tends to infinity through $\mathbb Z_{\ge0}$. Choose a homogeneous basis of $\mathfrak a$ and order it so that negative-degree elements come first, degree-zero elements come next, and positive-degree elements come last. If $\mathcal M=y_1\cdots y_d$ is a normally ordered PBW monomial, define
\begin{equation}
\deg_{>0}(\mathcal M)=\sum_{\deg y_j>0}\deg y_j.
\label{eq:positive-degree-of-pbw-monomial}
\end{equation}
Then, for every $q\in\mathbb Z$ and $L\in\mathbb Z_{\ge0}$,
\begin{equation}
T_LA(q)=\operatorname{span}\{\mathcal M:\deg\mathcal M=q,\ \deg_{>0}(\mathcal M)>L\}.
\label{eq:right-tail-pbw-description}
\end{equation}
Consequently,
\begin{equation}
\bigcap_{L\ge0}T_LA(q)=0,
\label{eq:right-tail-separatedness}
\end{equation}
and the canonical map $A\to\bigoplus_q\widehat A(q)$ is injective.
\end{lemma}

\begin{proof}
Since $T_{L+1}A(q)\subseteq T_LA(q)$, the natural quotient maps form the inverse system in \eqref{eq:right-tail-degreewise-completion}. These are vector-space quotients; the spaces $T_LA(q)$ need not be ideals.

The ordinary PBW theorem for Lie superalgebras gives the normally ordered supermonomials described in the statement as a basis of $A$; an odd basis element occurs with exponent at most one. Thus each element of $A$ has a finite expansion in this homogeneous monomial basis.

The function $\deg_{>0}(M)$ sums the degrees of the positive-degree factors in a normally ordered PBW monomial $M$. In the affine application, these are the positive loop indices. The total degree $\deg M$ sums the degrees of all factors, so positive and negative loop indices can cancel in that sum. The auxiliary function $\deg_{>0}$ describes the tail subspaces $T_LA(q)$ in \eqref{eq:right-tail-pbw-description}, which determine the standard completion through the quotient vector spaces in \eqref{eq:right-tail-degreewise-completion}.

We first prove the inclusion ``$\subseteq$'' in \eqref{eq:right-tail-pbw-description}. Since $A(s)$ is spanned by homogeneous normal PBW monomials, it is enough to take
\[
U\in A(q-s)
\]
and a normal monomial $V\in A(s)$, where $s>L$. Factor $V$ according to the signs of the degrees of its factors:
\begin{equation}
V=V_{<0}V_0V_{>0}.
\label{eq:ordered-monomial-degree-blocks}
\end{equation}
An empty factor in this decomposition is understood to be $1$. Put $r=\deg V_{>0}=\deg_{>0}(V)$. Since
\begin{equation}
s=\deg V=\deg V_{<0}+r\qquad\text{and}\qquad \deg V_{<0}\le0,
\label{eq:right-factor-degree-estimate}
\end{equation}
we have $r\ge s>L$.

Now write
\begin{equation}
UV=(UV_{<0}V_0)V_{>0}.
\label{eq:right-tail-monomial-factorization}
\end{equation}
Expand the parenthesized factor as a finite linear combination of normal PBW monomials. When $V_{>0}$ is appended to any one of these monomials, it is already to the right of all negative- and zero-degree factors. Therefore only positive-degree factors may have to be interchanged in order to restore the chosen order inside the positive block. If their degrees are $a,b>0$, their bracket lies in $\mathfrak a(a+b)$ by \eqref{eq:graded-lie-algebra}. Replacing their product by this bracket therefore preserves the sum of the positive degrees. Consequently every normal monomial occurring in $UV$ has
\begin{equation}
\deg_{>0}(\mathcal M)\ge\deg V_{>0}=r>L.
\label{eq:right-tail-positive-degree-bound}
\end{equation}
It follows that $A(q-s)A(s)$ is contained in the right-hand side of \eqref{eq:right-tail-pbw-description} whenever $s>L$. Summing over $s>L$ proves the required inclusion.

For the reverse inclusion, let $\mathcal M$ be a normal monomial of degree $q$ with $\deg_{>0}(\mathcal M)>L$. Because all its positive-degree factors occur at the right end, it has an exact factorization
\[
\mathcal M=\mathcal M_{\le0}\mathcal M_{>0},
\]
where $\mathcal M_{>0}$ is the product of all positive-degree factors. Its degree is $\deg_{>0}(\mathcal M)$, and hence
\[
\mathcal M\in A\bigl(q-\deg_{>0}(\mathcal M)\bigr)
 A\bigl(\deg_{>0}(\mathcal M)\bigr)\subseteq T_LA(q).
\]
This proves \eqref{eq:right-tail-pbw-description}.

Equation \eqref{eq:right-tail-pbw-description} also describes the topology: modulo $T_LA(q)$ one discards precisely those PBW monomials whose positive-degree sum is greater than $L$. Consequently, if $M_1,M_2,\ldots$ are normally ordered PBW monomials of the same total degree $q$ and $\deg_{>0}(M_r)$ tends to infinity, then $\sum_{r\ge1}M_r$ defines an element of $\widehat A(q)$. For every fixed $L$, all but finitely many summands belong to $T_LA(q)$, so the image in $A(q)/T_LA(q)$ is finite. This makes Ueda's infinite normally ordered quadratic root sums converge in the inverse-limit topology \eqref{eq:right-tail-degreewise-completion}.

The affine loop grading determines this completion, whereas the current filtration \eqref{eq:current-filtration-weights} determines the associated graded algebra in the PBW theorem. On the Yangian generators the loop grading, recorded again in \eqref{eq:yangian-affine-loop-grading}, is
\[
\deg_\ell h_{i,r}=0,\qquad
\deg_\ell x_{i,r}^{\pm}=
\begin{cases}
\pm1,&i=0,\\
0,&i\ne0,
\end{cases}
\]
Now let $0\ne a\in A(q)$. It is a finite linear combination of PBW monomials, so there is an $L$ for which every monomial $\mathcal M$ occurring in $a$ has $\deg_{>0}(\mathcal M)\le L$. Since distinct PBW monomials are linearly independent, \eqref{eq:right-tail-pbw-description} implies $a\notin T_LA(q)$. This proves \eqref{eq:right-tail-separatedness}.

Finally, the kernel of the canonical map $A(q)\to\widehat A(q)$ is exactly $\bigcap_{L\ge0}T_LA(q)$, so this map is injective for every $q$. An element of $A=\bigoplus_q A(q)$ has only finitely many homogeneous components; consequently the direct sum of these maps is injective. This is precisely the separatedness needed later: an equality between ordinary elements of $A$ that holds in the completion already holds in $A$ itself.
\end{proof}

\begin{lemma}\label{lem:real-root-cartan-identity}
Identify the current-degree-zero part of $\mathfrak s$ with
$\widehat{\g}_{\mathrm U}$ by Lemma~\ref{lem:current-degree-zero-identification}. If $\beta$ is a real affine root and
$0\ne x_\beta\in(\widehat{\g}_{\mathrm U})_\beta$, then, for $1\le i,j<N$,
\begin{equation}
 (\alpha_j,\beta)[H_{i,1},x_\beta]
   =(\alpha_i,\beta)[H_{j,1},x_\beta].
\label{eq:real-root-cartan-commutator-identity}
\end{equation}
\end{lemma}

\begin{proof}
Set $\varepsilon_1=\varepsilon_2=0$ directly in the defining presentation
of \cite[Definition~3.1]{UedaConstruction}. The resulting relations are homogeneous in current
degree and are precisely the enveloping relations of \eqref{eq:classical-current-relations}--\eqref{eq:classical-quartic-odd-relations}. Hence
\begin{equation}
 Y_{0,0}\bigl(\widehat{\mathfrak{sl}}(m|n)\bigr)=U(\mathfrak s)
 \qquad\text{as current-graded algebras}.
\label{eq:zero-parameter-classical-specialization}
\end{equation}
In the notation of \cite[Definition~4.5]{UedaConstruction}, the same substitution gives
\begin{equation}
 J(h_i)=H_{i,1}.
\label{eq:j-presentation-classical-cartan-generator}
\end{equation}

Apply \cite[Proposition~4.26 and its proof]{UedaConstruction} at these parameter values.
For each positive real affine root $\alpha$, the calculation gives the
real-root identity for $x_\alpha$ and its negative-root companion for
$x_{-\alpha}$, with coefficients paired with $\alpha$ and opposite scalar
signs. The loop grading used there is a grading of $\mathfrak s$.
Lemma~\ref{lem:degreewise-completion-separatedness} therefore makes the canonical map from $U(\mathfrak s)$ to that
completion injective. After \eqref{eq:j-presentation-classical-cartan-generator}, the difference between the two sides
of each identity is an ordinary element of $U(\mathfrak s)$; consequently
the identities hold in the uncompleted algebra:
\begin{equation}
 (\alpha_j,\alpha)[H_{i,1},x_{\pm\alpha}]
  -(\alpha_i,\alpha)[H_{j,1},x_{\pm\alpha}]
   =\pm c_{ij}^{\alpha}x_{\pm\alpha}
 \qquad\text{in }U(\mathfrak s).
\label{eq:specialized-real-root-identity}
\end{equation}
Here $c_{ij}^{\alpha}$ is the scalar supplied by the cited calculation at
$(\varepsilon_1,\varepsilon_2)=(0,0)$. The left-hand side has current
degree one, while the right-hand side has degree zero. Both homogeneous
components must vanish. Moreover, $x_{\pm\alpha}\ne0$ by \eqref{eq:enveloping-retraction}, so
$c_{ij}^{\alpha}=0$. This proves \eqref{eq:real-root-cartan-commutator-identity} for $\beta=\alpha$. For
$\beta=-\alpha$, multiply the negative-root identity by $-1$ and use
$(\alpha_k,-\alpha)=-(\alpha_k,\alpha)$.
\end{proof}

Choose, for every real affine root $\beta$, a nonzero root vector
$x_\beta$ in current degree zero. Its finite part is nonzero. The finite
Cartan form is nondegenerate, so choose a finite Cartan element $t_0$ with
$\beta(t_0)\ne0$. For $t=\sum_{i=1}^{N-1}t_iD_i\in\mathfrak h$, write
$H(t)=\sum_{i=1}^{N-1}t_iH_{i,1}$ and
$\beta(t)=\sum_{i=1}^{N-1}t_i(\alpha_i,\beta)$, and define recursively
\begin{equation}
 x_\beta[0]=x_\beta,
 \qquad
 x_\beta[r+1]=\frac{1}{\beta(t_0)}[H(t_0),x_\beta[r]].
\label{eq:higher-real-root-vector-recurrence}
\end{equation}

The recursion is homogeneous, so $x_\beta[r]$ has current degree $r$; the affine root $\beta$, including its loop component, remains fixed.

\begin{lemma}\label{lem:current-root-vector-recursion}
The element \eqref{eq:higher-real-root-vector-recurrence} is independent of $t_0$ and satisfies
\begin{equation}
[H(t),x_\beta[r]]=\beta(t)x_\beta[r+1]
\label{eq:higher-root-cartan-action}
\end{equation}
for every finite Cartan element $t$. If $\beta,\gamma$ have linearly independent finite parts, then
\begin{equation}
[x_\beta[r],x_\gamma[s]]=
\begin{cases}
c_{\beta,\gamma}x_{\beta+\gamma}[r+s],&[x_\beta,x_\gamma]=c_{\beta,\gamma}x_{\beta+\gamma}\ne0,\\
0,&[x_\beta,x_\gamma]=0.
\end{cases}
\label{eq:higher-real-root-brackets}
\end{equation}
\end{lemma}

\begin{proof}
Equation \eqref{eq:real-root-cartan-commutator-identity} gives \eqref{eq:higher-root-cartan-action} for $r=0$. All $H_{i,1}$ commute. Assuming \eqref{eq:higher-root-cartan-action} in degree $r$, we obtain
\begin{align*}
[H(t),x_\beta[r+1]]
&=\frac1{\beta(t_0)}[H(t_0),[H(t),x_\beta[r]]]\\
&=\frac{\beta(t)}{\beta(t_0)}[H(t_0),x_\beta[r+1]]
=\beta(t)x_\beta[r+2].
\end{align*}
This also proves independence of $t_0$.

For \eqref{eq:higher-real-root-brackets}, choose finite Cartan elements $t,q$ such that
\[
\gamma(t)=0,\quad\beta(t)\ne0,\quad\beta(q)=0,\quad\gamma(q)\ne0.
\]
Then
\[
x_\beta[r]=\beta(t)^{-r}(\operatorname{ad}H(t))^rx_\beta,
\qquad x_\gamma[s]=\gamma(q)^{-s}(\operatorname{ad}H(q))^sx_\gamma.
\]
The two adjoint operators commute; the first kills all $x_\gamma[s]$ and the second kills all $x_\beta[r]$. Applying the super-Jacobi identity repeatedly gives
\[
[x_\beta[r],x_\gamma[s]]
=\beta(t)^{-r}\gamma(q)^{-s}(\operatorname{ad}H(t))^r(\operatorname{ad}H(q))^s[x_\beta,x_\gamma],
\]
which is exactly \eqref{eq:higher-real-root-brackets}.
\end{proof}

For $a\ne b$ and $k\in\mathbb Z$, normalize the degree-zero root vector $x_{ab}(k)$ by
\begin{equation}
\rho(x_{ab}(k))=E_{ab}\otimes z^k.
\label{eq:normalized-real-root-images}
\end{equation}
This is unambiguous because the restriction of $\rho$ in current degree zero is the standard loop realization of $\gU$.

\begin{proposition}\label{prop:classical-presentation-isomorphism}
The map $\rho$ is an isomorphism
\begin{equation}
\sfr\xrightarrow{\sim}\Tplus.
\label{eq:rho-classical-isomorphism}
\end{equation}
\end{proposition}

\begin{proof}
Let $\rho_{\mathrm{st}}=\psi^{-1}\rho$. Define on Steinberg generators
\begin{equation}
\chi\bigl(F_{ab}(z^ku^r)\bigr)=x_{ab}(k)[r],
\qquad a\ne b,\ k\in\mathbb Z,\ r\ge0,
\label{eq:steinberg-to-classical-map}
\end{equation}
and extend linearly in the coefficient from $A_+$.

If $a,b,c$ are distinct, the finite parts of the roots $\varepsilon_a-\varepsilon_b$ and $\varepsilon_b-\varepsilon_c$ are linearly independent. Lemma~\ref{lem:current-root-vector-recursion}, with the normalization \eqref{eq:normalized-real-root-images}, gives
\begin{equation}
[x_{ab}(k)[r],x_{bc}(\ell)[s]]=x_{ac}(k+\ell)[r+s].
\label{eq:higher-root-steinberg-chain}
\end{equation}
If $b\ne c$ and $a\ne d$, the same lemma gives
\begin{equation}
[x_{ab}(k)[r],x_{cd}(\ell)[s]]=0
\label{eq:higher-root-steinberg-disjoint}
\end{equation}
whenever the finite root parts are linearly independent. We now classify the potentially dependent cases explicitly. Let $e_1,\ldots,e_N$ be the coordinate basis of $\mathbb C^N$. If the restrictions of $\varepsilon_a-\varepsilon_b$ and $\varepsilon_c-\varepsilon_d$ to $\mathfrak h$ are proportional by a nonzero scalar $\tau$, then
\begin{equation}
e_a-e_b-\tau(e_c-e_d)
=\mu\bigl((-1)^{p(1)},\ldots,(-1)^{p(N)}\bigr)
\label{eq:root-weight-proportionality}
\end{equation}
for some $\mu\in\mathbb C$. Indeed, the annihilator of the hyperplane $\sum_{\nu=1}^N(-1)^{p(\nu)}d_\nu=0$ is precisely the line spanned by $\bigl((-1)^{p(1)},\ldots,(-1)^{p(N)}\bigr)$. The union $\{a,b,c,d\}$ has at most four elements, whereas $N\ge5$; choose $e\notin\{a,b,c,d\}$. The $e$-th coordinate in \eqref{eq:root-weight-proportionality} is $0=\mu(-1)^{p(e)}$, and $(-1)^{p(e)}\in\{1,-1\}$, so $\mu=0$. Comparing the two nonzero coordinates now shows that either $\tau=1$ and $(c,d)=(a,b)$, or $\tau=-1$ and $(c,d)=(b,a)$. The second alternative is excluded by $b\ne c$ and $a\ne d$. Hence the only dependent case still to treat is $(c,d)=(a,b)$. Choose $e\notin\{a,b\}$ and write, using \eqref{eq:higher-root-steinberg-chain},
\[
x_{ab}(\ell)[s]=[x_{ae}(\ell)[s],x_{eb}(0)[0]].
\]
Both brackets of $x_{ab}(k)[r]$ with the two factors vanish by the already proved case. The super-Jacobi identity therefore proves \eqref{eq:higher-root-steinberg-disjoint} also in this last case.

Equations \eqref{eq:higher-root-steinberg-chain}--\eqref{eq:higher-root-steinberg-disjoint} are precisely the Steinberg relations
\[
[F_{ab}(f),F_{bc}(g)]=F_{ac}(fg),\qquad
[F_{ab}(f),F_{cd}(g)]=0\quad(b\ne c,\ a\ne d).
\]
Thus \eqref{eq:steinberg-to-classical-map} defines a homomorphism $\chi:\mathfrak{st}_{m|n}(A_+)\to\sfr$.

On every simple affine generator the recursion \eqref{eq:higher-root-cartan-action} is exactly the leading relation \eqref{eq:classical-current-relations}. Hence
\[
\chi\rho_{\mathrm{st}}(E^\pm_{i,r})=E^\pm_{i,r},\qquad
\chi\rho_{\mathrm{st}}(H_{i,r})=H_{i,r}.
\]
The Cartan equality follows from $H_{i,r}=[E^+_{i,r},E^-_{i,0}]$. Conversely, in $\Tplus$ the same Cartan recursion gives
\[
[\rho(H(t)),E_{ab}\otimes z^ku^r]
=(\varepsilon_a-\varepsilon_b)(t)E_{ab}\otimes z^ku^{r+1};
\]
the cocycle vanishes because one matrix is diagonal and the other is off diagonal. Therefore $\rho_{\mathrm{st}}\chi$ is the identity on all Steinberg generators. The two maps are inverse, so $\rho_{\mathrm{st}}$ and hence $\rho=\psi\rho_{\mathrm{st}}$ are injective. Together with the surjectivity of Lemma~\ref{lem:classical-map-extension}, this proves \eqref{eq:rho-classical-isomorphism}.
\end{proof}

From now on we identify $\sfr$ and $\Tplus$ by \eqref{eq:rho-classical-isomorphism}. Under this identification, \eqref{eq:canonical-classical-surjection} becomes the canonical epimorphism
\begin{equation}
\Phi:U(\Tplus)\twoheadrightarrow\operatorname{gr}_F\Y.
\label{eq:canonical-pbw-surjection}
\end{equation}

\section{Evaluation map}\label{sec:evaluation-map}

Evaluation homomorphisms are standard constructions in Yangian theory, going back to the finite-dimensional setting of Drinfeld \cite{Drinfeld85,Drinfeld88}. In ordinary affine type~$A$, an evaluation homomorphism into a completed affine enveloping algebra was constructed by Guay \cite[Section~7]{Guay2007}. Here we use the affine super Yangian evaluation homomorphism constructed by Ueda \cite[Theorem~5.1]{UedaConstruction}.

Put
\begin{equation}
\eta=m-n,\qquad c_0=\frac{n-m}{2}=-\frac\eta2,
\qquad\lambda=\frac1\eta.
\label{eq:evaluation-constants}
\end{equation}
All three numbers are nonzero. These are the parameters of the specialization used below.

Let $X(a)=X\otimes t^a$. To define the evaluation target $B$ explicitly, start with the affine Lie superalgebra used in \cite[Section~5, immediately before Theorem~5.1]{UedaConstruction}:
\begin{equation}
\widehat{\mathfrak{gl}}(m|n)
=\mathfrak{gl}(m|n)\otimes\mathbb C[t^{\pm1}]\oplus\mathbb Cc\oplus\mathbb Cz,
\label{eq:affine-general-linear-algebra}
\end{equation}
where $c$ and $z$ are central. Here $z$ denotes the central generator, while $t$ is the loop variable. The bracket is
\begin{equation}
[X(a),Y(b)]=[X,Y](a+b)
+a\delta_{a+b,0}\bigl(\str(XY)c+\str(X)\str(Y)z\bigr).
\label{eq:affine-general-linear-bracket}
\end{equation}
We define $B$ to be the standard degreewise completion of the enveloping algebra obtained after imposing $c=c_0$ and $z=1$. In this central quotient the bracket is
\begin{equation}
\begin{aligned}
[X(a),Y(b)]={}&[X,Y](a+b)\\
&+a\delta_{a+b,0}\bigl(c_0\str(XY)+\str(X)\str(Y)\bigr).
\end{aligned}
\label{eq:evaluation-central-quotient-bracket}
\end{equation}
The formula with the Kronecker factor in both central summands is printed in \cite[Section~2, p.~484]{UedaSurjectivity}; this factor is required by the loop grading. It is missing from the second central summand in the displayed formula preceding Theorem~5.1 of \cite{UedaConstruction}. For matrix units, \eqref{eq:evaluation-central-quotient-bracket} reads
\begin{equation}
\begin{aligned}
[E_{ab}(r),E_{cd}(s)]={}&\delta_{bc}E_{ad}(r+s)
-(-1)^{(p(a)+p(b))(p(c)+p(d))}\delta_{ad}E_{cb}(r+s)\\
&+r\delta_{r+s,0}\bigl(c_0\delta_{bc}\delta_{ad}(-1)^{p(a)}
+\delta_{ab}\delta_{cd}(-1)^{p(a)+p(c)}\bigr).
\end{aligned}
\label{eq:evaluation-matrix-unit-bracket}
\end{equation}
Indeed, $\str(E_{ab}E_{cd})=\delta_{bc}\delta_{ad}(-1)^{p(a)}$ and $\str(E_{ab})\str(E_{cd})=\delta_{ab}\delta_{cd}(-1)^{p(a)+p(c)}$. The matrix-unit bracket is given in Ueda \cite[Section~2, p.~484]{UedaSurjectivity}.

Explicitly, put
\[
B_0:=U\bigl(\widehat{\mathfrak{gl}}(m|n)\bigr)/(c-c_0,z-1)
\]
where $(c-c_0,z-1)$ denotes the two-sided ideal generated by the two displayed central elements. Give $B_0$ the loop grading induced by $\deg_\ell X(a)=a$ and $\deg_\ell c=\deg_\ell z=0$, and write $B_0(q)$ for its degree-$q$ subspace. For $q\in\mathbb Z$ and $M\in\mathbb Z_{\ge0}$, set
\begin{equation}
\begin{gathered}
T_MB_0(q)=\sum_{r>M}B_0(q-r)B_0(r),\qquad
\widehat B_0(q)=\varprojlim_M B_0(q)/T_MB_0(q),\\
B=\bigoplus_{q\in\mathbb Z}\widehat B_0(q).
\end{gathered}
\label{eq:evaluation-target-completion}
\end{equation}
The quotients in \eqref{eq:evaluation-target-completion} are vector-space quotients, and the inverse limits carry their natural degreewise topology. With the continuously extended product from $B_0$, this defines the completed associative superalgebra $B$. It is the degreewise completion used for the evaluation target in \cite[Definition~4.1 and Section~5]{UedaConstruction}.

Let $\Id=\sum_{\nu=1}^NE_{\nu\nu}$ be the identity matrix. Recall the convention $X(a)=X\otimes t^a$ introduced above. Accordingly, for $a\in\mathbb Z$, define
\begin{equation}
\Id(a):=\Id\otimes t^a
\label{eq:identity-loop-modes}
\end{equation}
$\Id(a)$ is the identity-matrix loop element of loop degree $a$, viewed first as an element of the affine Lie superalgebra and then as an element of its completed enveloping algebra $B$. In particular, $\Id(0)$ is the constant scalar-matrix Lie element $\Id\otimes1$, not the unit of the enveloping algebra. Since $m\ne n$,
\begin{equation}
\mathfrak{gl}(m|n)=\g\oplus\mathbb C\Id,\qquad
E_{aa}=E^0_{aa}+\frac{(-1)^{p(a)}}\eta\Id,\quad E^0_{aa}\in\g.
\label{eq:special-linear-scalar-splitting}
\end{equation}
Although $\Id$ commutes with every finite matrix, the central cocycle in the affine bracket gives a nontrivial bracket between identity-matrix loop elements of opposite loop degrees. Indeed, $\str(\Id)=\str(\Id^2)=\eta$ and $\str(\Id X)=\str(X)=0$ for $X\in\g$. Using $c_0=-\eta/2$ from \eqref{eq:evaluation-constants} in \eqref{eq:evaluation-central-quotient-bracket} therefore gives
\begin{equation}
[\Id(a),X(b)]=0\quad(X\in\g),\qquad
[\Id(a),\Id(b)]=a\delta_{a+b,0}\frac{\eta^2}{2}.
\label{eq:identity-mode-heisenberg-relations}
\end{equation}
Filter $B$ by the number of identity-matrix loop factors in a PBW monomial:
\begin{equation}
\deg_G\Id(a)=1,\qquad\deg_GX(a)=0\quad(X\in\g).
\label{eq:evaluation-target-filtration}
\end{equation}
Let $G_dB$ be the closed span of normal PBW monomials containing at most $d$ factors of the form $\Id(a)$, set $G_dB=0$ for $d<0$, put $\Bbd=\bigcup_{d\ge0}G_dB$, and write
\[
\sigma_d:G_dB\longrightarrow G_dB/G_{d-1}B,\qquad
j_a=\sigma_1(\Id(a))\in\operatorname{gr}_G\Bbd.
\]
The $j_a$ commute with one another and with the affine $\g$-part.

We record a cutoff lemma which will also justify the later use of completed tensor products.

\begin{lemma}\label{lem:continuous-pbw-coordinates}
Choose the normal PBW order in which negative-loop-degree basis elements
precede loop-degree-zero basis elements, which in turn precede
positive-loop-degree basis elements. For a normal monomial $M$, let
$\deg_{>0}(M)$ be the sum of its positive loop indices, as defined in \eqref{eq:positive-degree-of-pbw-monomial}.
Then
\begin{equation}
T_LB_0(q)=\operatorname{span}
\{M:\deg_\ell M=q,\ \deg_{>0}(M)>L\}.
\label{eq:normal-order-right-tail-description}
\end{equation}
Consequently every fixed PBW coefficient extends continuously to $B$,
and these coefficients separate $B$. The same assertions hold for finite
tensor powers with total loop grading, for $\operatorname{gr}_G\Bbd$,
and for its algebraic finite tensor powers. After fixing the number of
identity-matrix loop factors, passage to the $G$-symbol is therefore valid
coefficient by coefficient.
\end{lemma}

\begin{proof}
Choose a homogeneous basis of $\widehat{\mathfrak{gl}}(m|n)$ containing
the two central basis elements. Order the central elements first and
the remaining basis elements by the signs of their loop degrees. The
ordinary PBW theorem makes the enveloping algebra free over the
polynomial algebra in the two central elements. After imposing $c=c_0$
and $z=1$, the normal monomials in the remaining basis elements therefore
form a basis of $B_0$.

Take $y\in B_0(q-r)$ and $z\in B_0(r)$ with $r>L$. By linearity it is
enough to take $y$ and $z$ to be normal PBW monomials. Factor the second
monomial according to loop degree:
\begin{equation}
z=z_{<0}z_0z_{>0}.
\label{eq:normal-order-loop-degree-blocks}
\end{equation}
Since
\begin{equation}
r=\deg_\ell z=\deg_\ell z_{<0}+\deg_\ell z_{>0}
\qquad\text{and}\qquad \deg_\ell z_{<0}\le0,
\label{eq:normal-order-right-factor-degree}
\end{equation}
we have $\deg_\ell z_{>0}\ge r>L$. Regroup the product as
\begin{equation}
yz=(yz_{<0}z_0)z_{>0}.
\label{eq:normal-order-product-factorization}
\end{equation}
Expand the parenthesized factor in the normal PBW basis. The appended
factor $z_{>0}$ is already to the right of every factor of nonpositive
loop degree. Restoring the chosen order can therefore require only
interchanges among factors of positive loop degree. The bracket of
elements of positive loop degrees $a$ and $b$ has loop degree $a+b$;
replacing their product by this bracket preserves the sum of positive
loop degrees. Hence every normal monomial occurring in $yz$ has
$\deg_{>0}$ at least $\deg_\ell z_{>0}>L$. This proves the inclusion
``$\subseteq$'' in \eqref{eq:normal-order-right-tail-description}.

Conversely, if a normal monomial has $\deg_{>0}(M)>L$, split off the
product of all its positive-loop-degree factors as the right factor.
That factor has loop degree $\deg_{>0}(M)>L$, so $M\in T_LB_0(q)$.
This proves \eqref{eq:normal-order-right-tail-description}. The coefficient of a fixed $M$ therefore factors
through every cutoff $L\ge\deg_{>0}(M)$, while \eqref{eq:normal-order-right-tail-description} also shows that
all such coefficients together separate the inverse limit.

The same normal-ordering estimate applies to a finite tensor power,
where $\deg_{>0}$ is computed by summing the positive degrees over all
tensor factors. Finally, \eqref{eq:identity-mode-heisenberg-relations} shows that reordering two identity-matrix
loop factors can only lower their number by two and never raise it.
Hence extraction of a fixed $G$-homogeneous PBW coefficient is compatible
with each finite cutoff. This proves the last assertion.
\end{proof}

Introduce the formal current
\begin{equation}
\mathcal J(z)=\sum_{q\in\mathbb Z}j_qz^{-q}.
\label{eq:identity-symbol-generating-series}
\end{equation}
We use the homogeneous central basis \eqref{eq:central-differential-basis}, whose derivation was given there.

In \eqref{eq:theta-current-formula}--\eqref{eq:theta-central-basis-images}, choose the lifts
$:\mathrm{Id}(q_1)\cdots\mathrm{Id}(q_r)X(b):$ of
$j_{q_1}\cdots j_{q_r}X(b)$, omitting $X(b)$ for central terms.
Colons select the corresponding normally ordered PBW monomial of
Lemma~\ref{lem:continuous-pbw-coordinates}. Extend this convention linearly. Sum these normally ordered
lifts in $B$ before taking their $G$-symbols; \eqref{eq:normal-order-right-tail-description} ensures local
finiteness at every cutoff.

Define $\Theta:\Tplus\to\operatorname{gr}_G\Bbd$ by
\begin{equation}
\Theta(Xz^au^r)=\lambda^r\sum_{q_1,\ldots,q_r\in\mathbb Z}
j_{q_1}\cdots j_{q_r}X(a-q_1-\cdots-q_r),
\label{eq:theta-current-formula}
\end{equation}
\begin{equation}
\Theta\!\left(\overline{P\,dz+Q\,du}\right)
=c_0\operatorname{Res}_{z=0}
\left(P(z,\lambda\mathcal J(z))
+Q(z,\lambda\mathcal J(z))\frac{d}{dz}(\lambda\mathcal J(z))\right)dz.
\label{eq:theta-differential-formula}
\end{equation}
For $r=0$ the sum in \eqref{eq:theta-current-formula} is understood as the single empty tuple.
Formula \eqref{eq:theta-differential-formula} is the residue of the pullback of a differential form
under $u\mapsto\lambda\mathcal J(z)$; hence it kills exact forms.
In particular,
\begin{equation}
\begin{aligned}
\Theta(\kappa_{0,0})&=c_0,\\
\Theta(\kappa_{k,r})&=c_0\lambda^r[z^{-k}]\mathcal J(z)^r.
\end{aligned}
\label{eq:theta-central-basis-images}
\end{equation}

\begin{lemma}\label{lem:evaluation-symbol-embedding}
The normally ordered lifts of \eqref{eq:theta-current-formula}--\eqref{eq:theta-central-basis-images} converge in \eqref{eq:evaluation-target-completion}.
Their $G$-symbols define an injective Lie superalgebra homomorphism
$\Theta$, graded by the $u$-degree on $\Tplus$ and the $G$-degree in the target.
\end{lemma}

\begin{proof}

Fix a cutoff $M$. A summand of \eqref{eq:theta-current-formula} can survive only if
\[
q_i\le M\quad(1\le i\le r),\qquad
a-q_1-\cdots-q_r\le M.
\]
Consequently $a-rM\le q_i\le M$, so only finitely many tuples survive.
For \eqref{eq:theta-central-basis-images}, the conditions are $q_i\le M$ and
$q_1+\cdots+q_r=k$, which similarly give
$k-(r-1)M\le q_i\le M$. Thus the normally ordered lifts are locally
finite in each quotient of \eqref{eq:evaluation-target-completion}.

It remains to verify explicitly that $\Theta$ preserves the bracket.
By bilinearity it is enough to take homogeneous $X,Y\in\g$ and monomials
\begin{equation}
f=z^au^r,\qquad g=z^bu^s,\qquad
a,b\in\mathbb Z,\quad r,s\in\mathbb Z_{\ge0}.
\label{eq:monomial-current-test-functions}
\end{equation}
After the substitution $u=\lambda\mathcal J(z)$, put
\begin{equation}
f_J(z)=z^a(\lambda\mathcal J(z))^r,\qquad
g_J(z)=z^b(\lambda\mathcal J(z))^s
\label{eq:formal-current-substitutions}
\end{equation}
and write
\begin{equation}
f_J(z)=\sum_{k\in\mathbb Z}f_kz^k,\qquad
g_J(z)=\sum_{l\in\mathbb Z}g_lz^l.
\label{eq:formal-current-laurent-expansions}
\end{equation}
With the empty-product convention when $r=0$ or $s=0$, the coefficients are
\begin{equation}
\begin{aligned}
f_k&=\lambda^r\sum_{q_1+\cdots+q_r=a-k}j_{q_1}\cdots j_{q_r},\\
g_l&=\lambda^s\sum_{v_1+\cdots+v_s=b-l}j_{v_1}\cdots j_{v_s}.
\end{aligned}
\label{eq:formal-current-laurent-coefficients}
\end{equation}
Consequently, \eqref{eq:theta-current-formula} is equivalently
\begin{equation}
\Theta(X\otimes f)=\sum_{k\in\mathbb Z}f_kX(k),\qquad
\Theta(Y\otimes g)=\sum_{l\in\mathbb Z}g_lY(l).
\label{eq:theta-laurent-expansions}
\end{equation}

All calculations below are coefficientwise: after a cutoff and an output
PBW coefficient are fixed, Lemma~\ref{lem:continuous-pbw-coordinates} and the preceding local-finiteness
estimate leave only finitely many summands.

We first justify treating $f_k$ and $g_l$ as commuting central coefficients
in $\operatorname{gr}_G\Bbd$. Before taking the $G$-symbol, each $j_q$
is represented by $\mathrm{Id}(q)$. Equations \eqref{eq:identity-mode-heisenberg-relations} and \eqref{eq:evaluation-target-filtration} give
\begin{equation}
[\mathrm{Id}(q),\mathrm{Id}(v)]
=q\delta_{q+v,0}\frac{\eta^2}{2}\in G_0B,
\qquad [\mathrm{Id}(q),X(k)]=0\quad(X\in\g).
\label{eq:identity-filtration-commutators}
\end{equation}

The first commutator in \eqref{eq:identity-filtration-commutators} has $G$-degree zero, two less than the product of its factors. Hence all $j_q$ commute in $\operatorname{gr}_G$, and they commute with every loop element $X(k)$, $X\in\g$. Formula \eqref{eq:theta-differential-formula} therefore sends every differential class to a central element; brackets involving the central summand of $\Tplus$ are preserved.

Since $X,Y\in\mathfrak{sl}(m|n)$, $\str(X)=\str(Y)=0$, so the second cocycle in \eqref{eq:evaluation-central-quotient-bracket} vanishes. In $G$-degree $r+s$ that bracket and \eqref{eq:theta-laurent-expansions} give
\begin{equation}
\begin{aligned}
[\Theta(X\otimes f),\Theta(Y\otimes g)]
={}&\sum_{k,l\in\mathbb Z}f_kg_l[X,Y](k+l)\\
&+c_0\str(XY)\sum_{k,l\in\mathbb Z}k\delta_{k+l,0}f_kg_l.
\end{aligned}
\label{eq:theta-expanded-bracket}
\end{equation}
Commutators among representatives of the form $\Id(q)$ have $G$-degree at most $r+s-2$ by \eqref{eq:identity-filtration-commutators}, so they do not occur in \eqref{eq:theta-expanded-bracket}.

For the noncentral part, the Cauchy product gives
\begin{equation}
\sum_{k+l=n}f_kg_l
=[z^n]\bigl(f_J(z)g_J(z)\bigr)
=[z^n]\bigl(z^{a+b}(\lambda\mathcal J(z))^{r+s}\bigr),
\label{eq:formal-current-coefficient-convolution}
\end{equation}
and therefore
\begin{equation}
\sum_{k,l\in\mathbb Z}f_kg_l[X,Y](k+l)=\Theta\bigl([X,Y]\otimes fg\bigr).
\label{eq:theta-noncentral-bracket}
\end{equation}
For the central part,
\begin{equation}
\operatorname{Res}_{z=0}g_J(z)d_zf_J(z)
=\sum_{k,l\in\mathbb Z}k\delta_{k+l,0}f_kg_l
=\sum_{k\in\mathbb Z}kf_kg_{-k}.
\label{eq:formal-current-residue-pairing}
\end{equation}
A term multiplied by $r$ is understood to be absent when $r=0$. Differentiating the substituted monomial explicitly gives
\begin{equation}
\begin{aligned}
d_zf_J(z)=\Bigl(&az^{a-1}(\lambda\mathcal J(z))^r\\
&+rz^a(\lambda\mathcal J(z))^{r-1}\frac d{dz}(\lambda\mathcal J(z))\Bigr)dz,
\end{aligned}
\label{eq:formal-current-chain-rule}
\end{equation}
and hence
\begin{equation}
\begin{aligned}
g_J(z)d_zf_J(z)=\Bigl(&az^{a+b-1}(\lambda\mathcal J(z))^{r+s}\\
&+rz^{a+b}(\lambda\mathcal J(z))^{r+s-1}\frac d{dz}(\lambda\mathcal J(z))\Bigr)dz.
\end{aligned}
\label{eq:formal-current-cocycle-product}
\end{equation}
On the other hand, the differential in $A_+$ is
\begin{equation}
g\,df=az^{a+b-1}u^{r+s}\,dz+rz^{a+b}u^{r+s-1}\,du.
\label{eq:monomial-differential-cocycle}
\end{equation}
Applying \eqref{eq:theta-differential-formula} to the last formula and comparing with \eqref{eq:formal-current-cocycle-product} yields
\begin{equation}
\Theta\bigl(\overline{g\,df}\bigr)
=c_0\operatorname{Res}_{z=0}g_J(z)d_zf_J(z).
\label{eq:theta-cocycle-residue}
\end{equation}
Thus the central part of \eqref{eq:theta-expanded-bracket} is
\begin{equation}
c_0\str(XY)\sum_{k\in\mathbb Z}kf_kg_{-k}
=\str(XY)\Theta\bigl(\overline{g\,df}\bigr).
\label{eq:theta-central-bracket}
\end{equation}
Combining the noncentral and central calculations,
\begin{equation}
[\Theta(X\otimes f),\Theta(Y\otimes g)]
=\Theta\bigl([X,Y]\otimes fg+\str(XY)\overline{g\,df}\bigr).
\label{eq:theta-lie-homomorphism}
\end{equation}
The expression in parentheses is the bracket \eqref{eq:universal-central-extension-bracket}, so $\Theta$ is a Lie homomorphism.

It remains to prove injectivity as a linear map. First split a finite element of $\Tplus$ by its $u$-degree, which becomes its $G$-degree. In fixed degree $r$, apply the continuous PBW-coordinate functionals indexed by the monomials $j_0^rX(a)$, where $X$ ranges over a fixed homogeneous basis of $\g$ and $a\in\mathbb Z$. On the image of $\sum_aX_a\otimes z^au^r$, these coordinates are precisely the coordinates of $\lambda^r\sum_aj_0^rX_a(a)$. The direct-sum decomposition $\g\otimes\mathbb C[t^{\pm1}]=\bigoplus_{a\in\mathbb Z}\g\otimes t^a$, together with the ordinary PBW theorem at the fixed level, shows that its scalar PBW coordinates vanish only when every $X_a$ vanishes. The image of the central summand is a series involving only the commuting elements $j_q$ and has no factor $X(a)$; it therefore cannot contribute to these coordinates.

For the central part of fixed $u$-degree $r\ge1$, set $P_{k,r}=[z^{-k}]\mathcal J(z)^r$. The grading $\operatorname{wt}(j_q)=q$ gives $P_{k,r}$ weight $k$, so different $k$ cannot cancel. Moreover, $P_{0,r}$ contains $j_0^r$, while for $k\ne0$ the coefficient of $j_kj_0^{r-1}$ is $r$; hence every $P_{k,r}$ is nonzero. Finally, $\kappa_{0,0}$ maps to the nonzero scalar $c_0$. The PBW coordinates used in these projections are legitimate and separating by Lemma~\ref{lem:continuous-pbw-coordinates}. Therefore $\Theta$ is injective.
\end{proof}

By the universal property of the enveloping superalgebra, we use the same symbol $\Theta$ for the induced unital algebra homomorphism
\begin{equation}
U(\Tplus)\longrightarrow\operatorname{gr}_G\Bbd.
\label{eq:theta-enveloping-homomorphism}
\end{equation}
The tensor-power separation argument for this enveloping map is given in Lemma~\ref{lem:tensor-power-separation}.

We use the evaluation homomorphism constructed in \cite[Theorem~5.1, formulas~(5.2)--(5.4)]{UedaConstruction}, specialized to
\(\varepsilon_1=\varepsilon_2=\tfrac12\), \(z=1\), and
\(c=c_0=(n-m)/2\). The level condition in that theorem is satisfied by
this specialization. For \(\alpha_{\mathrm{ev}}\in\mathbb C\), write
\begin{equation}
  \ev_{\alpha_{\mathrm{ev}}}\colon Y_\hbar\longrightarrow B.
  \label{eq:evaluation-homomorphism}
\end{equation}

\begin{proposition}\label{prop:evaluation-leading-symbols}
The map~\eqref{eq:evaluation-homomorphism} sends \(F_rY_\hbar\) into \(G_rB^{\mathrm{bd}}\), and
\begin{equation}
\begin{aligned}
  \sigma_r\ev_{\alpha_{\mathrm{ev}}}(x^\pm_{i,r})
    &=\Theta\rho(E^\pm_{i,r}),\\
  \sigma_r\ev_{\alpha_{\mathrm{ev}}}(h_{i,r})
    &=\Theta\rho(H_{i,r}).
\end{aligned}
\label{eq:evaluation-principal-symbols}
\end{equation}
\end{proposition}

\begin{proof}
For \(r=0\), the values in \cite[(5.2)]{UedaConstruction} coincide term by term with
\(\Theta\rho\), using \eqref{eq:finite-generator-assignment}--\eqref{eq:affine-generator-assignment}: formula~\eqref{eq:theta-current-formula} gives
\(\Theta(Xz^a)=X(a)\), while \(\Theta(\kappa_{0,0})=c_0\).

We calculate the degree-one symbols explicitly. From \eqref{eq:special-linear-scalar-splitting},
\begin{equation}
  E_{aa}(q)=E^0_{aa}(q)
     +\frac{(-1)^{p(a)}}{\eta}\Id(q),
  \qquad \lambda=\eta^{-1}.
  \label{eq:evaluation-diagonal-splitting}
\end{equation}
Consider first \(x^+_{i,1}\), where \(1\leq i<N\), using the finite-node
case of \cite[(5.3)]{UedaConstruction}. Its linear term has \(G\)-degree zero. In the first
quadratic sum, with matrix factors \(E_{ik}(-s)E_{k,i+1}(s)\) and
\(1\leq k\leq i\), only \(k=i\) gives a diagonal factor. Its part
containing \(\Id\) is
\begin{equation}
  \frac{((-1)^{p(i)})^2}{\eta}
       \sum_{s\geq0}\Id(-s)E_{i,i+1}(s)
   =\lambda\sum_{s\geq0}\Id(-s)E_{i,i+1}(s).
  \label{eq:positive-finite-diagonal-evaluation-term}
\end{equation}
In the second quadratic sum, with matrix factors
\(E_{ik}(-s-1)E_{k,i+1}(s+1)\) and \(i+1\leq k\leq N\), only
\(k=i+1\) gives a diagonal factor. Its coefficient is again
\(\eta^{-1}\), and its part containing \(\Id\) is
\begin{equation}
  \lambda\sum_{s\geq0}\Id(s+1)E_{i,i+1}(-s-1).
  \label{eq:positive-finite-shifted-diagonal-term}
\end{equation}
Here \(\Id(s+1)\) commutes with the off-diagonal loop element by \eqref{eq:identity-mode-heisenberg-relations}.
The ranges \(q=-s\leq0\) and \(q=s+1\geq1\) partition \(\mathbb Z\).
All remaining summands have \(G\)-degree zero, and no quadratic summand
contains two diagonal factors. Thus the evaluation has \(G\)-degree at
most one and
\begin{equation}
  \sigma_1\ev_{\alpha_{\mathrm{ev}}}(x^+_{i,1})
   =\lambda\sum_{q\in\mathbb Z}j_qE_{i,i+1}(-q)
   =\Theta\rho(E^+_{i,1}).
  \label{eq:positive-finite-evaluation-symbol}
\end{equation}

\begingroup\emergencystretch=2em
For \(x^-_{i,1}\), use the finite-node case of \cite[(5.4)]{UedaConstruction}. The first
quadratic sum has factors \(E_{i+1,k}(-s)E_{ki}(s)\), with
\(1\leq k\leq i\); its diagonal index is \(k=i\), in the second
factor. The second sum has factors
\(E_{i+1,k}(-s-1)E_{ki}(s+1)\), with \(i+1\leq k\leq N\);
its diagonal index is \(k=i+1\), in the first factor. Including the
outer factor \((-1)^{p(i)}\) in the cited formula, the coefficients of
the resulting \(\Id\)-terms are respectively\par\endgroup
\begin{equation}
  \frac{((-1)^{p(i)})^3}{\eta}=(-1)^{p(i)}\lambda,
  \qquad
  \frac{(-1)^{p(i)}((-1)^{p(i+1)})^2}{\eta}
      =(-1)^{p(i)}\lambda.
  \label{eq:negative-finite-evaluation-coefficients}
\end{equation}
The corresponding ranges are \(q=s\geq0\) and
\(q=-s-1\leq-1\). They partition \(\mathbb Z\), and every other term
has \(G\)-degree zero. Consequently,
\begin{equation}
  \sigma_1\ev_{\alpha_{\mathrm{ev}}}(x^-_{i,1})
   =(-1)^{p(i)}\lambda
       \sum_{q\in\mathbb Z}j_qE_{i+1,i}(-q)
   =\Theta\rho(E^-_{i,1}).
  \label{eq:negative-finite-evaluation-symbol}
\end{equation}

\begingroup\emergencystretch=2em
At the affine vertex, the quadratic sum in the \(i=0\) case of
\cite[(5.3)]{UedaConstruction} has factors \(E_{Nk}(-s)E_{k1}(s+1)\).
Only \(k=N\) and \(k=1\) produce diagonal factors, respectively
\(E_{NN}(-s)\) and \(E_{11}(s+1)\). After \eqref{eq:evaluation-diagonal-splitting}, their coefficients
are both \(\lambda\), and their ranges are \(q=-s\leq0\) and
\(q=s+1\geq1\). Hence\par\endgroup
\begin{equation}
  \sigma_1\ev_{\alpha_{\mathrm{ev}}}(x^+_{0,1})
   =\lambda\sum_{q\in\mathbb Z}j_qE_{N1}(1-q)
   =\Theta\rho(E^+_{0,1}).
  \label{eq:positive-affine-evaluation-symbol}
\end{equation}
Similarly, the quadratic sum in the \(i=0\) case of \cite[(5.4)]{UedaConstruction} has
factors \(E_{1k}(-s-1)E_{kN}(s)\), with an outer minus sign.
The diagonal indices \(k=1\) and \(k=N\) give the ranges
\(q=-s-1\leq-1\) and \(q=s\geq0\), both with coefficient
\(-\lambda\). Thus
\begin{equation}
  \sigma_1\ev_{\alpha_{\mathrm{ev}}}(x^-_{0,1})
   =-\lambda\sum_{q\in\mathbb Z}j_qE_{1N}(-1-q)
   =\Theta\rho(E^-_{0,1}).
  \label{eq:negative-affine-evaluation-symbol}
\end{equation}
The linear terms and all quadratic terms without a diagonal factor
have \(G\)-degree zero, so these computations also establish the
filtration bound at the affine vertex.

The defining relation \(h_{i,1}=[x^+_{i,1},x^-_{i,0}]\) and
multiplicativity of \(\ev_{\alpha_{\mathrm{ev}}}\) now give
\begin{equation}
\begin{aligned}
  \sigma_1\ev_{\alpha_{\mathrm{ev}}}(h_{i,1})
    &=[\Theta\rho(E^+_{i,1}),\Theta\rho(E^-_{i,0})]\\
    &=\Theta\rho(H_{i,1}).
\end{aligned}
\label{eq:cartan-degree-one-evaluation-symbol}
\end{equation}
For \(i=0\), the central term is visible explicitly:
\begin{equation}
\begin{aligned}
 \left[\lambda\sum_{q\in\mathbb Z}j_qE_{N1}(1-q),
          -E_{1N}(-1)\right]
  &=\lambda\sum_{q\in\mathbb Z}j_q(-E_{11}-E_{NN})(-q)
       +c_0\lambda j_0.
\end{aligned}
\label{eq:affine-cartan-central-evaluation-term}
\end{equation}
The cocycle is nonzero only for \(q=0\), and
\(\str(E_{N1}(-E_{1N}))=1\). The right-hand side is
\begin{equation}
  \Theta\rho(H_{0,1})
    =\Theta\bigl((-E_{11}-E_{NN})\otimes u+\kappa_{0,1}\bigr),
  \label{eq:affine-cartan-theta-image}
\end{equation}
because \(\Theta(\kappa_{0,1})=c_0\lambda j_0\).

We next derive the recurrence needed for all higher current degrees
directly from the defining relations. Set
\begin{equation}
  \widetilde h_{k,1}=h_{k,1}-\frac12h_{k,0}^2.
  \label{eq:corrected-degree-one-cartan}
\end{equation}
We claim that
\begin{equation}
  [\widetilde h_{k,1},x^\pm_{j,r}]
     =\pm a_{kj}x^\pm_{j,r+1}.
  \label{eq:yangian-higher-current-recurrence}
\end{equation}
Taking the first current index to be zero in \eqref{eq:cartan-current-difference-relation}, with
\(\hbar=1\), gives
\begin{equation}
  [h_{k,1},x^\pm_{j,r}]-[h_{k,0},x^\pm_{j,r+1}]
       =\pm\frac{a_{kj}}2\{h_{k,0},x^\pm_{j,r}\}.
  \label{eq:degree-one-cartan-difference-relation}
\end{equation}
The third relation in \eqref{eq:yangian-basic-relations}, with current index \(r+1\), is
\begin{equation}
  [h_{k,0},x^\pm_{j,r+1}]=\pm a_{kj}x^\pm_{j,r+1}.
  \label{eq:zero-mode-cartan-action-higher-current}
\end{equation}
Substitution into \eqref{eq:degree-one-cartan-difference-relation} yields
\begin{equation}
  [h_{k,1},x^\pm_{j,r}]
    =\pm a_{kj}x^\pm_{j,r+1}
       \pm\frac{a_{kj}}2\{h_{k,0},x^\pm_{j,r}\}.
  \label{eq:degree-one-cartan-action-expanded}
\end{equation}
Since \(h_{k,0}\) is even, the super-Leibniz rule introduces no extra
sign. Using \eqref{eq:yangian-basic-relations} again gives
\begin{equation}
\begin{aligned}
 [h_{k,0}^2,x^\pm_{j,r}]
   &=h_{k,0}[h_{k,0},x^\pm_{j,r}]
       +[h_{k,0},x^\pm_{j,r}]h_{k,0}\\
   &=\pm a_{kj}\bigl(h_{k,0}x^\pm_{j,r}
                     +x^\pm_{j,r}h_{k,0}\bigr)\\
   &=\pm a_{kj}\{h_{k,0},x^\pm_{j,r}\}.
\end{aligned}
\label{eq:cartan-square-commutator}
\end{equation}
Consequently, \eqref{eq:corrected-degree-one-cartan}, \eqref{eq:degree-one-cartan-action-expanded}, and \eqref{eq:cartan-square-commutator} imply
\[
 [\widetilde h_{k,1},x^\pm_{j,r}]
   =[h_{k,1},x^\pm_{j,r}]
      -\frac12[h_{k,0}^2,x^\pm_{j,r}]
   =\pm a_{kj}x^\pm_{j,r+1},
\]
proving \eqref{eq:yangian-higher-current-recurrence}.

For every \(j\in I\), choose a finite node \(k\) with
\(a_{kj}\ne0\): take \(k=j\) for a non-isotropic finite node and a
finite neighbour for \(j=0\) or \(j=m\). Formula~\eqref{eq:yangian-higher-current-recurrence} gives
\begin{equation}
 \ev_{\alpha_{\mathrm{ev}}}(x^\pm_{j,r+1})
   =\frac{\pm1}{a_{kj}}
       [\ev_{\alpha_{\mathrm{ev}}}(\widetilde h_{k,1}),
        \ev_{\alpha_{\mathrm{ev}}}(x^\pm_{j,r})].
 \label{eq:higher-root-evaluation-recurrence}
\end{equation}
Because \(\ev_{\alpha_{\mathrm{ev}}}(h_{k,0}^2)\in G_0B\),
formula~\eqref{eq:cartan-degree-one-evaluation-symbol} gives
\begin{equation}
 \sigma_1\ev_{\alpha_{\mathrm{ev}}}(\widetilde h_{k,1})
       =\Theta\rho(H_{k,1}).
 \label{eq:corrected-cartan-evaluation-symbol}
\end{equation}
Suppose the assertion for \(x^\pm_{j,r}\) is known in degree \(r\).
The filtered multiplication and commutator in the target show from
\eqref{eq:higher-root-evaluation-recurrence} that its next evaluation lies in \(G_{r+1}B\). Taking its symbol,
using \eqref{eq:corrected-cartan-evaluation-symbol}, and applying the Lie homomorphism \(\Theta\rho\), we obtain
\begin{equation}
\begin{aligned}
 \sigma_{r+1}\ev_{\alpha_{\mathrm{ev}}}(x^\pm_{j,r+1})
   &=\frac{\pm1}{a_{kj}}
        [\Theta\rho(H_{k,1}),\Theta\rho(E^\pm_{j,r})]\\
   &=\Theta\rho(E^\pm_{j,r+1}).
\end{aligned}
\label{eq:higher-root-evaluation-symbol}
\end{equation}
Here \([H_{k,1},E^\pm_{j,r}]=\pm a_{kj}E^\pm_{j,r+1}\) in
\(\mathfrak s\), by \eqref{eq:classical-current-relations}. This proves both the symbol formula and the
filtration bound for every \(x^\pm_{j,r}\). Finally,
\begin{equation}
 h_{i,r}=[x^+_{i,r},x^-_{i,0}]
 \label{eq:cartan-as-opposite-root-bracket}
\end{equation}
proves the same assertions for every \(h_{i,r}\). Multiplicativity then
gives \(\ev_{\alpha_{\mathrm{ev}}}(F_rY_\hbar)\subseteq
G_rB^{\mathrm{bd}}\), as required.
\end{proof}

\section{Coproduct in the affine super Yangian}
\label{sec:coproduct}
The coproduct construction for ordinary affine Yangians was developed by Guay, Nakajima, and Wendlandt \cite[Theorem~4.9 and Proposition~5.18]{GNW}. The corresponding coproduct for the affine super Yangian considered here was constructed by Ueda \cite[Theorem~4.3, formula~(4.4)]{UedaConstruction}, following that framework. We first specify the completed tensor product in which its values lie.

\subsection{The completed tensor product}
\label{subsec:completed-tensor-product}
For this section write $Y=Y_{\hbar}$. The completion used in the coproduct comes from the affine loop grading:
\begin{equation}
\deg_{\ell}h_{i,r}=0,\qquad
\deg_{\ell}x^{\pm}_{i,r}=\begin{cases}\pm1,&i=0,\\0,&i\ne0.\end{cases}
\label{eq:yangian-affine-loop-grading}
\end{equation}
For $d\ge1$, put $A_d=Y^{\otimes d}$ with total loop grading and define
\begin{equation}
T_MA_d(q)=\sum_{s>M}A_d(q-s)A_d(s),\qquad
\widehat{Y^{\otimes d}}=\bigoplus_{q\in\mathbb Z}\varprojlim_M A_d(q)/T_MA_d(q).
\label{eq:completed-tensor-power}
\end{equation}
These are vector-space quotients; the tails need not be ideals. Give $A_d$ the total current filtration
\begin{equation}
F^{\mathrm{tot}}_pA_d=\sum_{\substack{p_1,\ldots,p_d\ge0\\p_1+\cdots+p_d\le p}}
F_{p_1}Y\otimes\cdots\otimes F_{p_d}Y
\label{eq:total-current-tensor-filtration}
\end{equation}
and denote its closure in \eqref{eq:completed-tensor-power} by $\widehat F_pA_d$. Equivalently,
\begin{equation}
(\widehat F_pA_d)(q)=\varprojlim_M\operatorname{im}\bigl(F^{\mathrm{tot}}_pA_d(q)\longrightarrow A_d(q)/T_MA_d(q)\bigr).
\label{eq:closed-current-tensor-filtration}
\end{equation}

Continuity of multiplication gives
\begin{equation}
  \widehat F_aA_d\,\widehat F_bA_d\subseteq\widehat F_{a+b}A_d,
  \qquad
  [\widehat F_aA_d,\widehat F_bA_d]\subseteq\widehat F_{a+b}A_d.
  \label{eq:completed-filtration-multiplicativity}
\end{equation}

We use the coassociative, loop-graded coproduct constructed in
\cite[Definition~4.1, Theorem~4.3, formula~(4.4)]{UedaConstruction}, specialized to
$\varepsilon_1=\varepsilon_2=\tfrac12$, and denote it by
\begin{equation}
  \Delta\colon Y\longrightarrow\widehat{Y\otimes Y}.
  \label{eq:completed-yangian-coproduct}
\end{equation}
Its tensor extensions and iterates are taken in the total-loop-degree
completions \eqref{eq:completed-tensor-power}. All factors in the cited correction have current degree zero.
The root sum converges in the stated completion: if a root vector
$x^a_\alpha$ has positive loop degree $q$, its paired tensor satisfies
\begin{equation}
  x^a_{-\alpha}\otimes x^a_\alpha
  =(x^a_{-\alpha}\otimes1)(1\otimes x^a_\alpha)
  \in A_2(-q)A_2(q).
  \label{eq:root-sum-tail-convergence}
\end{equation}
Here the root vectors denote their current-degree-zero images in $Y$, as in
the cited formula. The displayed tensor vanishes modulo $T_MA_2(0)$ for
$q>M$, and only finitely many affine roots, counted with multiplicity,
occur at each fixed loop degree. Thus the sum is Cauchy, and its limit
belongs to the closed subspace $\widehat F_0A_2$.

Put $\widetilde h_{i,1}=h_{i,1}-\tfrac12h_{i,0}^2$.
The cited formula gives the following local notation and filtration property:
\begin{equation}
  \begin{aligned}
    P_2(a)&=a\otimes1+1\otimes a,\\
    R_i&:=\Delta(\widetilde h_{i,1})-P_2(\widetilde h_{i,1})
          \in\widehat F_0A_2.
  \end{aligned}
  \label{eq:coproduct-cartan-correction}
\end{equation}

\begin{lemma}
\label{lem:coproduct-leading-term}
For $A_{i,r}=x^\pm_{i,r}$ or $h_{i,r}$,
\begin{equation}
  \Delta(A_{i,r})=P_2(A_{i,r})+Q_{i,r},
  \qquad Q_{i,r}\in\widehat F_{r-1}A_2,
  \label{eq:coproduct-lower-current-correction}
\end{equation}
where $\widehat F_{-1}A_2=0$. If
\begin{equation}
  P_d(a)=\sum_{\nu=1}^d
    1^{\otimes(\nu-1)}\otimes a\otimes1^{\otimes(d-\nu)},
  \label{eq:primitive-tensor-sum}
\end{equation}
then
\begin{equation}
  \Delta^{(d-1)}(A_{i,r})
  =P_d(A_{i,r})\pmod{\widehat F_{r-1}A_d}.
  \label{eq:iterated-coproduct-principal-part}
\end{equation}
Thus the top current-degree operation is the ordinary primitive coproduct
$\Delta_0$.
\end{lemma}

Here and below,
\begin{equation}
  \Delta_0\colon U(\Tplus)\longrightarrow U(\Tplus)^{\otimes2},
  \qquad\Delta_0(x)=x\otimes1+1\otimes x\quad(x\in\Tplus),
  \label{eq:classical-primitive-coproduct}
\end{equation}
is the unique algebra homomorphism determined by primitive values on the
Lie generators; $\Delta_0^{(d-1)}$ denotes its iterated form.

\begin{proof}
For $r=0$ the assertion follows from \cite[formula~(4.4)]{UedaConstruction}. The same formula
gives the assertion for $h_{i,1}$. For fixed $j$, choose $k$ with
$a_{kj}\ne0$. In every formula below, the upper or lower signs are taken
consistently. Applying $\Delta$ to
\begin{equation}
  [\widetilde h_{k,1},x^\pm_{j,r}]
  =\pm a_{kj}x^\pm_{j,r+1}
  \label{eq:coproduct-current-recurrence}
\end{equation}
at $r=0$ gives
\begin{equation}
  \Delta(x^\pm_{j,1})
  =P_2(x^\pm_{j,1})
    \pm\frac{1}{a_{kj}}[R_k,P_2(x^\pm_{j,0})].
  \label{eq:coproduct-degree-one-root-formula}
\end{equation}
The correction lies in $\widehat F_0A_2$ by \eqref{eq:completed-filtration-multiplicativity} and \eqref{eq:coproduct-cartan-correction}.
All brackets and tensor products here use the super sign conventions.

Suppose
$\Delta(x^\pm_{j,r})=P_2(x^\pm_{j,r})+Q^\pm_{j,r}$ with
$Q^\pm_{j,r}\in\widehat F_{r-1}A_2$.
Applying $\Delta$ to \eqref{eq:coproduct-current-recurrence} and subtracting the primitive part gives the
exact recurrence
\begin{equation}
  \begin{aligned}
    \pm a_{kj}Q^\pm_{j,r+1}
    ={}&[P_2(\widetilde h_{k,1}),Q^\pm_{j,r}]\\
       &+[R_k,P_2(x^\pm_{j,r})]+[R_k,Q^\pm_{j,r}].
  \end{aligned}
  \label{eq:coproduct-correction-recurrence}
\end{equation}
The three terms have filtration degrees at most $r$, $r$, $r-1$,
respectively, by \eqref{eq:completed-filtration-multiplicativity}. Therefore
$Q^\pm_{j,r+1}\in\widehat F_rA_2$. Finally,
\begin{equation}
  \Delta(h_{i,r})-P_2(h_{i,r})
  =[Q^+_{i,r},P_2(x^-_{i,0})]\in\widehat F_{r-1}A_2.
  \label{eq:coproduct-cartan-filtration-bound}
\end{equation}
This proves \eqref{eq:coproduct-lower-current-correction} and also
$\Delta(F_pY)\subseteq\widehat F_pA_2$.
Continuity, coassociativity of the cited coproduct, and induction on $d$
then give \eqref{eq:iterated-coproduct-principal-part}.
\end{proof}

\subsection{Continuity of tensor evaluation}\label{subsec:tensor-evaluation-continuity}

Let
\begin{equation}
B_d=\widehat{B_0^{\otimes d}}
\label{eq:completed-evaluation-tensor-target}
\end{equation}
be the analogous total-loop-degree completion; thus $B_1=B$. On $B_0^{\otimes d}$ let $G_p^{\mathrm{alg}}$ be the span of tensor PBW monomials containing at most $p$ factors of the form $\Id(a)$ in total, and define its closure by
\begin{equation}
(G_p^{\mathrm{tot}}B_d)(q)
=\varprojlim_M\operatorname{im}\bigl(G_p^{\mathrm{alg}}(q)\longrightarrow
B_0^{\otimes d}(q)/T_MB_0^{\otimes d}(q)\bigr).
\label{eq:total-identity-tensor-filtration}
\end{equation}
Write
\begin{equation}
\sigma_p^{(d)}:G_p^{\mathrm{tot}}B_d\longrightarrow
G_p^{\mathrm{tot}}B_d/G_{p-1}^{\mathrm{tot}}B_d.
\label{eq:tensor-evaluation-symbol-map}
\end{equation}
There is a canonical graded linear map
\begin{equation}
(\operatorname{gr}_G\Bbd)^{\otimes d}\longrightarrow
\bigoplus_{p\ge0}G_p^{\mathrm{tot}}B_d/G_{p-1}^{\mathrm{tot}}B_d
\label{eq:graded-algebraic-tensor-inclusion}
\end{equation}
which, on homogeneous representatives $b_\nu\in G_{p_\nu}\Bbd$, sends
\begin{equation}
\sigma_{p_1}(b_1)\otimes\cdots\otimes\sigma_{p_d}(b_d)
\longmapsto\sigma_{p_1+\cdots+p_d}^{(d)}(b_1\otimes\cdots\otimes b_d).
\label{eq:graded-tensor-inclusion-on-symbols}
\end{equation}
This is well defined because replacing one representative by an element of $G_{p_\nu-1}\Bbd$ lowers the total $G$-degree. It is injective: a nonzero element of the algebraic tensor product on the left has a nonzero tensor PBW coordinate, and Lemma~\ref{lem:continuous-pbw-coordinates} shows that the same coordinate is defined and has the same value on the right. We henceforth use \eqref{eq:graded-algebraic-tensor-inclusion} to regard $(\operatorname{gr}_G\Bbd)^{\otimes d}$ as a graded subspace of the right-hand side.

If $f:C\to D$ is a loop-graded algebra homomorphism, then
\begin{equation}
f(C(q-s)C(s))\subseteq D(q-s)D(s),\qquad
f(T_MC(q))\subseteq T_MD(q),
\label{eq:loop-graded-homomorphism-tail-continuity}
\end{equation}
so it extends continuously to standard completions. Ueda's coproduct and evaluation assignments are loop-graded. The coproduct tails are locally finite by \eqref{eq:root-sum-tail-convergence}. In the quadratic evaluation sums \cite[(5.3)--(5.4)]{UedaConstruction}, the pairs of factor degrees are
\begin{equation}
(-s,s),\quad(-s-1,s+1),\quad(-s,s+1),\quad(-s-1,s).
\label{eq:quadratic-evaluation-loop-degree-pairs}
\end{equation}
The last two pairs are the affine-node cases. In every pair the right-hand factor has loop degree $s$ or $s+1$. Hence, for a fixed cutoff $M$, every summand with $s>M$ belongs to $T_M$, and only the finitely many indices $0\le s\le M$ can survive modulo $T_M$. This explicitly proves the required local finiteness and continuity.

Proposition~\ref{prop:evaluation-leading-symbols} and multiplicativity show that, after the canonical map from the algebraic tensor product $B^{\otimes d}$ to its total-loop-degree completion $B_d$,
\begin{equation}
\ev_{\alpha_{\mathrm{ev}}}^{\otimes d}(F_p^{\mathrm{tot}}A_d)
\subseteq G_p^{\mathrm{tot}}B_d.
\label{eq:tensor-evaluation-filtration-bound}
\end{equation}
Indeed, each evaluated factor of current filtration degree $p_\nu$ belongs to $G_{p_\nu}\Bbd$ and is, at every cutoff, represented by a finite linear combination of PBW monomials with at most $p_\nu$ factors of the form $\Id(a)$. A tensor product with $p_1+\cdots+p_d\le p$ is therefore represented at every cutoff by tensor PBW monomials with at most $p$ such factors in total. This is exactly membership in the right-hand side of \eqref{eq:tensor-evaluation-filtration-bound}.

The space $G_p^{\mathrm{tot}}B_d$ is closed by its inverse-limit definition \eqref{eq:total-identity-tensor-filtration}, and tensor evaluation is continuous by the preceding cutoff calculation. Taking the closure of $F_p^{\mathrm{tot}}A_d$ in \eqref{eq:closed-current-tensor-filtration} therefore proves
\begin{equation}
\widehat{\ev_{\alpha_{\mathrm{ev}}}^{\otimes d}}(\widehat F_pA_d)
\subseteq G_p^{\mathrm{tot}}B_d.
\label{eq:completed-tensor-evaluation-filtration-bound}
\end{equation}
Consequently,
\begin{equation}
\Psi_d=\widehat{\ev_{\alpha_{\mathrm{ev}}}^{\otimes d}}\circ\Delta^{(d-1)}:
Y\longrightarrow B_d
\label{eq:iterated-coproduct-evaluation}
\end{equation}
is well defined. Lemma~\ref{lem:coproduct-leading-term} and \eqref{eq:completed-tensor-evaluation-filtration-bound} also give
\begin{equation}
\Psi_d(F_pY)\subseteq G_p^{\mathrm{tot}}B_d.
\label{eq:tensor-evaluation-current-filtration}
\end{equation}
We use fixed continuous PBW-coordinate functionals to compute the graded symbols below.

\begin{lemma}\label{lem:iterated-evaluation-symbols}
Let $v\in U(\Tplus)$ be homogeneous of current degree $r$, and choose a homogeneous polynomial lift $\widetilde v$ in the Yangian generators. For every tensor PBW coefficient $c_M$ whose indexing monomial $M$ has total $G$-degree $r$,
\begin{equation}
c_M\bigl(\sigma_r^{(d)}\Psi_d(\widetilde v)\bigr)
=c_M\bigl(\Theta^{\otimes d}\Delta_0^{(d-1)}(v)\bigr).
\label{eq:tensor-evaluation-leading-coefficients}
\end{equation}
In particular, the right side is independent of the chosen lift.
\end{lemma}

\begin{proof}
Fix $M$. By Lemma~\ref{lem:continuous-pbw-coordinates}, $c_M$ is continuous. Expand a fixed polynomial representative of $v$, apply the coproduct and evaluation formulas, and then apply $c_M$.

Continuity of multiplication in the standard degreewise completion and the normal-ordering estimate \eqref{eq:normal-order-right-tail-description} imply that, for this fixed output coefficient, only bounded values of $\deg_{>0}$ can contribute. At every such bound the sums in the evaluation formulas and the affine-root sums in the coproduct contain only finitely many contributing loop indices. Thus, after applying $c_M$, the calculation is an ordinary finite calculation.

\begingroup\tolerance=1000\emergencystretch=3em
Lemma~\ref{lem:coproduct-leading-term} shows that the component of current degree $r$ in the iterated coproduct is $\Delta_0^{(d-1)}(v)$. Proposition~\ref{prop:evaluation-leading-symbols} shows that evaluation of this component has $G$-symbol $\Theta^{\otimes d}\Delta_0^{(d-1)}(v)$. Every remaining term has $G$-degree at most $r-1$. Applying $c_M$ proves \eqref{eq:tensor-evaluation-leading-coefficients}.\par\endgroup
\end{proof}

\begin{lemma}\label{lem:tensor-power-separation}
Let $\mathfrak a$ be a Lie superalgebra, $C$ an associative superalgebra, and $\vartheta:\mathfrak a\to C$ an injective even Lie homomorphism. Then the maps
\begin{equation}
U(\mathfrak a)\longrightarrow C^{\otimes d},\qquad
v\longmapsto\vartheta^{\otimes d}\Delta_0^{(d-1)}(v),\qquad
d=0,1,2,\ldots,
\label{eq:tensor-powers-joint-separation}
\end{equation}
are jointly injective; for $d=0$ the map is the counit.
\end{lemma}

\begin{proof}
Let $0\ne v$ have top PBW degree $d$; the scalar case is detected by the counit. For $S\subseteq\{1,\ldots,d\}$ apply the map with $|S|$ tensor factors and insert units in the complementary positions; call the result $T_S(v)$. Consider
\begin{equation}
D_d(v)=\sum_{S\subseteq\{1,\ldots,d\}}(-1)^{d-|S|}T_S(v).
\label{eq:inclusion-exclusion-tensor-operator}
\end{equation}
Expanding the primitive coproduct assigns every factor of a PBW monomial to a tensor position. Inclusion--exclusion kills all assignments which do not use all $d$ positions. Hence \eqref{eq:inclusion-exclusion-tensor-operator} vanishes on PBW degree $<d$, while on the top symbol it is
\begin{equation}
x_1\cdots x_d\longmapsto\sum_{\sigma\in\mathfrak S_d}
\epsilon(\sigma;x)\vartheta(x_{\sigma(1)})\otimes\cdots\otimes\vartheta(x_{\sigma(d)}).
\label{eq:super-symmetrization-top-symbol}
\end{equation}
Here $\epsilon(\sigma;x)$ is the Koszul sign obtained by permuting the homogeneous elements $x_1,\ldots,x_d$. Before applying $\vartheta^{\otimes d}$, \eqref{eq:super-symmetrization-top-symbol} is the canonical super-symmetrization $s_d:S_{\mathrm s}^d(\mathfrak a)\to\mathfrak a^{\otimes d}$. If $q_d$ is the quotient map in the other direction, then $q_ds_d=d!\operatorname{id}$. Thus $s_d$ is injective in characteristic zero. The map $\vartheta^{\otimes d}$ is injective as a linear map. Consequently \eqref{eq:inclusion-exclusion-tensor-operator} detects the nonzero top PBW symbol of $v$, proving joint injectivity.
\end{proof}

\begin{proof}[Proof of Theorem~\ref{thm:pbw}]
\emph{Injectivity of $\Phi$.} Surjectivity of \eqref{eq:canonical-pbw-surjection} was established in Proposition~\ref{prop:graded-surjection}, after the identification of Proposition~\ref{prop:classical-presentation-isomorphism}. It remains to prove injectivity. Suppose that $0\ne v\in\ker\Phi$. Since $\Phi$ is graded, we may take $v$ homogeneous of current degree $r$. Choose a homogeneous polynomial lift $\widetilde v$ in the Yangian generators. The equality $\Phi(v)=0$ means exactly that
\begin{equation}
\widetilde v\in F_{r-1}\Y.
\label{eq:kernel-lift-filtration-drop}
\end{equation}
The standard augmentations send every Lie or Yangian generator to zero and commute with $\Phi$. Hence the counit of $v$ is zero. The separating tensor power in the next paragraph may therefore be taken with $d\ge1$.

Apply Lemma~\ref{lem:tensor-power-separation} to the injective Lie map $\Theta:\Tplus\to\operatorname{gr}_G\Bbd$. There is a $d$ such that
\begin{equation}
w=\Theta^{\otimes d}\Delta_0^{(d-1)}(v)\ne0.
\label{eq:nonzero-separated-classical-tensor}
\end{equation}
By Lemma~\ref{lem:continuous-pbw-coordinates}, tensor PBW coordinates separate $(\operatorname{gr}_G\Bbd)^{\otimes d}$. Since $w$ is homogeneous of total $G$-degree $r$, some tensor PBW monomial $M$ of total $G$-degree $r$ satisfies $c_M(w)\ne0$.

On the other hand, \eqref{eq:kernel-lift-filtration-drop} and \eqref{eq:tensor-evaluation-current-filtration} imply that $\Psi_d(\widetilde v)$ belongs to $G_{r-1}^{\mathrm{tot}}B_d$. Its degree-$r$ symbol is therefore zero. Lemma~\ref{lem:iterated-evaluation-symbols} now gives
\[
0=c_M\bigl(\sigma_r^{(d)}\Psi_d(\widetilde v)\bigr)=c_M(w),
\]
contradicting the choice of $c_M$. Thus $\ker\Phi=0$. Together with Proposition~\ref{prop:graded-surjection}, this proves \eqref{eq:pbw-graded-isomorphism}.

\emph{Spanning (completeness) of the ordered lifts.} The ordinary PBW theorem for Lie superalgebras gives the ordered supermonomial basis of $U(\Tplus)$. Let $y\in F_r\Y$. By surjectivity of $\Phi$, the degree-$r$ symbol of $y$ is a finite linear combination of the symbols of ordered lifted monomials of total current degree $r$. Subtracting this combination from $y$ leaves an element of $F_{r-1}\Y$. Induction on $r$, with $F_{-1}\Y=0$, gives a finite expansion of $y$ in the monomials \eqref{eq:ordered-filtered-lift-monomials}. Thus these monomials span $\Y$.

\emph{Linear independence of the ordered lifts.} Suppose that a finite linear combination of the monomials \eqref{eq:ordered-filtered-lift-monomials} vanishes, and choose the maximal total current degree occurring in that combination. Passing to the corresponding graded component and using injectivity of $\Phi$ gives a linear relation among distinct classical ordered PBW monomials of that degree. Their coefficients therefore vanish. Descending induction on current degree proves that every coefficient vanishes. Consequently, \eqref{eq:ordered-filtered-lift-monomials} is a basis.
\end{proof}

\appendix
\renewcommand{\theequation}{\thesection.\arabic{equation}}
\section{An explicit order and the PBW argument for \texorpdfstring{$\widehat{\mathfrak{sl}}(3|2)$}{affine sl(3|2)}}
\label{app:order}
We specialize to $m=3$, $n=2$ and $\hbar=1$. Thus
\[
p(1)=p(2)=p(3)=0,\qquad p(4)=p(5)=1,\qquad
\eta=\lambda=1,\qquad c_0=-\tfrac12.
\]
The ordered basis below makes the construction and the separation argument
of Theorem~\ref{thm:pbw} explicit in this case.

\subsection{A homogeneous basis and its order}\label{subsec:appendix-homogeneous-basis-order}
Put
\begin{equation}\label{app:diagonal}
\begin{aligned}
D_1&=E_{11}-E_{22},&D_2&=E_{22}-E_{33},\\
D_3&=E_{33}+E_{44},&D_4&=-E_{44}+E_{55}.
\end{aligned}
\end{equation}
The twenty off-diagonal matrix units and these four diagonal elements form
a homogeneous basis $\mathcal B_{\rm fin}$ of $\mathfrak{sl}(3|2)$.
There are twelve even and twelve odd elements: $E_{ab}$ is odd exactly
when one index lies in $\{1,2,3\}$ and the other in $\{4,5\}$.
Choose the following order of finite types and assign the labels $t=1,\ldots,24$:
\[
\underbrace{E_{ab}\ (a>b)}_{\text{lexicographic in }(a,b)},\quad
D_1,D_2,D_3,D_4,\quad
\underbrace{E_{ab}\ (a<b)}_{\text{lexicographic in }(a,b)}.
\]
Assign the label $t=25$ to the central type. By \eqref{eq:universal-central-extension-model} and \eqref{eq:central-differential-basis}, a basis of
$\Tplus$ is
\begin{equation}\label{app:basis}
\begin{aligned}
\mathcal B={}&\{Xz^ku^r:X\in\mathcal B_{\rm fin},\ k\in\mathbb Z,\ r\ge0\}\\
&\sqcup\{c\}\sqcup\{\kappa_{k,r}:k\in\mathbb Z,\ r\ge1\},\\
c&=\kappa_{0,0}=\overline{z^{-1}dz},\qquad
\kappa_{k,r}=\overline{z^{k-1}u^r\,dz}.
\end{aligned}
\end{equation}
All central elements are even. Give $Xz^ku^r$ and $\kappa_{k,r}$ the key
\begin{equation}\label{app:orderkey}
\bigl(\operatorname{sgn}(k),\ r,\ |k|,\ t\bigr),
\qquad -1<0<1,
\end{equation}
and give $c$ the key $(0,0,0,25)$. Compare keys lexicographically.
This specifies a total order on every element of $\mathcal B$.
Negative loop degrees precede zero loop degree, which precedes positive
loop degrees; within each sign, current degree, absolute loop degree and
the finite-type label resolve the order. For example,
\[
E_{43}z^{-1}u<E_{34}z,\qquad
E_{43}<E_{34}u,\qquad D_1u<D_1z.
\]
The ordinary super PBW basis of $U(\Tplus)$ consists of the finite products
\begin{equation}\label{app:monomials}
b_1^{d_1}\cdots b_s^{d_s},\qquad b_1<\cdots<b_s,
\end{equation}
with $d_i\ge0$ for even $b_i$ and $d_i\in\{0,1\}$ for odd $b_i$.
Their current degree is $\sum_i d_i\deg_u b_i$.

\subsection{Reordering, central terms and filtered lifts}\label{subsec:appendix-reordering-and-lifts}
For $K=k+\ell$ and $R=r+s>0$, the exact differential $d(z^Ku^R)$ gives
\begin{equation}\label{app:cocycle}
\begin{aligned}
\overline{z^Ku^{R-1}\,du}&=-\frac K R\kappa_{K,R},\\
\overline{z^\ell u^s\,d(z^ku^r)}
&=\frac{ks-r\ell}{r+s}\kappa_{k+\ell,r+s}.
\end{aligned}
\end{equation}
For $r=s=0$ the latter class is $k\delta_{k+\ell,0}c$.
Consequently \eqref{eq:universal-central-extension-bracket} computes every bracket between the basis elements in
\eqref{app:basis}. The rule in the enveloping algebra is
\[
ab=(-1)^{|a||b|}ba+[a,b].
\]
In the odd example $A=E_{34}z$ and $B=E_{43}z^{-1}u$, we have
$B<A$, $[E_{34},E_{43}]=D_3$ and $\str(E_{34}E_{43})=1$. Therefore
\begin{equation}\label{app:oddreorder}
[A,B]=D_3u+\kappa_{0,1},\qquad
AB=-BA+D_3u+\kappa_{0,1}.
\end{equation}
The product $BA$ is ordered, and the correction consists of basis elements
of current degree one. For an even example, $\str(D_1^2)=2$ and
$\overline{u\,dz}=\kappa_{1,1}$, so
\begin{equation}\label{app:evenreorder}
[D_1z,D_1u]=2\kappa_{1,1},\qquad
(D_1z)(D_1u)=(D_1u)(D_1z)+2\kappa_{1,1}.
\end{equation}
These calculations exhibit both the super sign and the differential-form
central correction used in the classical presentation.

Use the identification $\rho:\mathfrak s\to\Tplus$ of Proposition~\ref{prop:classical-presentation-isomorphism} to
express each element of $\mathcal B$ as a finite homogeneous Lie polynomial
in the classical generators. Replacing these generators by the Yangian
generators gives a filtered lift $\widehat b\in F_{\deg_u b}Y_\hbar$.
For instance,
\begin{equation}\label{app:lifts}
\widehat{E_{34}u}=x^+_{3,1},\quad
\widehat{E_{43}}=x^-_{3,0},\quad
\widehat{\kappa_{0,r}}=\sum_{i=0}^4h_{i,r}\quad(r\ge0).
\end{equation}
The last equality specifies a choice of lift: under $\rho$, the sum of the
five classical Cartan generators is $\kappa_{0,r}$.
The products $\widehat b_1^{d_1}\cdots\widehat b_s^{d_s}$ with the order and
exponents in \eqref{app:monomials} are the Yangian basis asserted by
Theorem~\ref{thm:pbw}. The following calculations illustrate its independence proof.

\subsection{Evaluation symbols and tensor separation}\label{subsec:appendix-evaluation-and-separation}
With $\mathcal J(z)=\sum_qj_qz^{-q}$, formulas \eqref{eq:theta-current-formula}--\eqref{eq:theta-central-basis-images} become
\begin{equation}\label{app:theta}
\begin{aligned}
\Theta(Xz^ku^r)
&=\sum_{q_1,\ldots,q_r\in\mathbb Z}
 j_{q_1}\cdots j_{q_r}X(k-q_1-\cdots-q_r),\\
\Theta(c)&=-\tfrac12,\qquad
\Theta(\kappa_{k,r})=-\tfrac12[z^{-k}]\mathcal J(z)^r.
\end{aligned}
\end{equation}
The sums use the normally ordered lifts prescribed before \eqref{eq:theta-current-formula}, with the
order in Lemma~\ref{lem:continuous-pbw-coordinates} on the evaluation algebra. For the degree-zero matrix
and identity modes, fix the identity mode before the matrix modes within
loop degree zero. In degree one,
\begin{equation}\label{app:symbols}
\begin{aligned}
\sigma_1\ev(x^+_{3,1})&=\sum_qj_qE_{34}(-q),\\
\sigma_1\ev(h_{0,1})
&=\sum_qj_q(-E_{11}-E_{55})(-q)-\tfrac12j_0,\\
\Theta(\kappa_{k,1})&=-\tfrac12j_k.
\end{aligned}
\end{equation}
Here $\ev=\operatorname{ev}_{\alpha_{\rm ev}}$, with arbitrary evaluation
parameter. For a cutoff $M$, a term in the first sum can survive only if
$q\le M$ and $-q\le M$, so only $-M\le q\le M$ contribute. This is the
local finiteness argument of Lemma~\ref{lem:evaluation-symbol-embedding} in a concrete instance.
The image of the bracket in \eqref{app:oddreorder} is
\[
\Theta(D_3u+\kappa_{0,1})=\sum_qj_qD_3(-q)-\tfrac12j_0,
\]
and that of $2\kappa_{1,1}$ in \eqref{app:evenreorder} is $-j_1$.
Thus the central corrections survive in the evaluation symbols with their
specified signs.

The coefficient of $j_0^rX(k)$ in $\Theta(Xz^ku^r)$ is $1$.
For the central image the coefficient of $j_0^r$ when $k=0$ is $-1/2$;
when $k\ne0$, the coefficient of $j_kj_0^{r-1}$ is $-r/2$.
The central images involve only the $j$ modes, and their weights
$\operatorname{wt}(j_q)=q$ distinguish the index $k$. These coordinates
show explicitly how Lemma~\ref{lem:evaluation-symbol-embedding} separates finite linear combinations of
our Lie basis.

Write $T_d=\Theta^{\otimes d}\Delta_0^{(d-1)}$ for $d\ge1$ and let $T_0$
be the counit. Primitivity of $c$ gives
\begin{equation}\label{app:centralpowers}
T_d(P(c))=P(-d/2)\,1^{\otimes d}\quad(d\ge1).
\end{equation}
For $P(c)=c(c+1/2)$, we obtain $T_0(P)=T_1(P)=0$ and
$T_2(P)=\tfrac12\,1\otimes1$. In general, if $P$ has degree $\ell$ and
leading coefficient $a_\ell$, the inclusion--exclusion operator in
Lemma~\ref{lem:tensor-power-separation} gives
\begin{equation}\label{app:finitedifference}
D_\ell(P(c))=
\left(\sum_{k=0}^\ell(-1)^{\ell-k}\binom\ell kP(-k/2)\right)
1^{\otimes\ell}
=a_\ell\ell!(-\tfrac12)^\ell1^{\otimes\ell}.
\end{equation}
Thus tensor separation detects every nonzero polynomial in the
current-degree-zero central generator.

For a mixed odd example, take $f=E_{43}$ and $e=E_{34}u$. Since $f<e$,
$v=fe$ is an ordered monomial of current degree one and PBW length two.
The primitive super coproduct gives
\begin{equation}\label{app:oddseparation}
\begin{aligned}
D_2(v)&=\Theta(f)\otimes\Theta(e)-\Theta(e)\otimes\Theta(f)\\
&=E_{43}(0)\otimes\sum_qj_qE_{34}(-q)
 -\sum_qj_qE_{34}(-q)\otimes E_{43}(0).
\end{aligned}
\end{equation}
The minus sign is the Koszul sign for interchanging two odd factors.
The tensor PBW coefficient indexed by
\begin{equation}\label{app:coordinate}
M=E_{43}(0)\otimes j_0E_{34}(0)
\end{equation}
is $1$, and its total $G$-degree is one. In $T_2(v)$ the two additional
terms $\Theta(fe)\otimes1$ and $1\otimes\Theta(fe)$ have zero $M$-coefficient,
so $c_M(T_2(v))=1$ as well. For the lift
$\widetilde v=x^-_{3,0}x^+_{3,1}$, Lemma~\ref{lem:iterated-evaluation-symbols} therefore yields
\begin{equation}\label{app:yangiancheck}
c_M\bigl(\sigma^{(2)}_1\Psi_2(\widetilde v)\bigr)=1.
\end{equation}
Lemma~\ref{lem:coproduct-leading-term} and Proposition~\ref{prop:evaluation-leading-symbols} give
$\Psi_2(F_0Y_\hbar)\subseteq G^{\rm tot}_0B_2$.
Equation \eqref{app:yangiancheck} consequently proves
$\widetilde v\notin F_0Y_\hbar$.

The general proof uses the same steps for a finite homogeneous polynomial
$v\in U(\Tplus)$ of current degree $r$. Its maximal PBW length $\ell$
selects the operator $D_\ell$ of Lemma~\ref{lem:tensor-power-separation}, whose top symbol is the injective
super-symmetrization. Some $T_d(v)$ therefore has a nonzero tensor PBW
coefficient of total $G$-degree $r$. For a putative kernel element the
counit vanishes, so $d$ can be chosen positive. If $\Phi(v)=0$, a polynomial lift lies
in $F_{r-1}Y_\hbar$; Lemma~\ref{lem:iterated-evaluation-symbols} identifies that coefficient with the
corresponding coefficient of $\sigma_r^{(d)}\Psi_d(\widetilde v)$, whereas
Lemma~\ref{lem:coproduct-leading-term} and Proposition~\ref{prop:evaluation-leading-symbols} put its evaluation in $G^{\rm tot}_{r-1}B_d$.
This is the contradiction proving injectivity of $\Phi$. Finally, the
ordinary super PBW basis \eqref{app:monomials} and induction on current
filtration degree give the ordered lifted basis for the chosen order.

\par\medskip
\noindent\textbf{Acknowledgement.} This work was supported by the Russian Science Foundation, grant \mbox{26-11-00115}, and carried out at the Center of Fundamental Mathematics, Moscow Institute of Physics and Technology.

\end{document}